\documentclass[a4paper]{article}

\usepackage{amsmath}
\usepackage{amsthm}
\usepackage{amssymb}
\usepackage{amscd}
\usepackage{ifthen} 
\usepackage{a4wide}
\usepackage{epsfig,graphics,graphicx}
\usepackage{array,tabularx}
\usepackage{url}

\newenvironment{denseBibitem}[1]
{\bibitem{#1}\begin{minipage}[t]{\textwidth}}
{\end{minipage}}

\providecommand{\defnterm}[1]{{\em #1}}

\newtheorem{theorem}{Theorem}[section]
\newtheorem{definition}[theorem]{Definition}
\newtheorem{proposition}[theorem]{Proposition}
\newtheorem{corollary}[theorem]{Corollary}
\newtheorem{lemma}[theorem]{Lemma}
\newtheorem{fact}[theorem]{Remark}
\newtheorem{exemplu}[theorem]{Example}
\newtheorem{exercise}{Exercise}
\newtheorem{notation}[theorem]{Notation}
\newtheorem{remark}[theorem]{Remark}
\newtheorem{problem}[theorem]{Problem}

\newcommand{\bdfn}{\begin{definition}}
\newcommand{\edfn}{\end{definition}}
\newcommand{\bthm}{\begin{theorem}}
\newcommand{\ethm}{\end{theorem}}
\newcommand{\bprop}{\begin{proposition}}
\newcommand{\eprop}{\end{proposition}}
\newcommand{\bcor}{\begin{corollary}}
\newcommand{\ecor}{\end{corollary}}
\newcommand{\blem}{\begin{lemma}}
\newcommand{\elem}{\end{lemma}}
\newcommand{\bfact}{\begin{fact}}
\newcommand{\efact}{\end{fact}}
\newcommand{\bex}{\begin{exemplu}\begin{rm}}
\newcommand{\eex}{\end{rm}\end{exemplu}}
\newcommand{\bxc}{\begin{exercise}}
\newcommand{\exc}{\end{exercise}}
\newcommand{\bntn}{\begin{notation}}
\newcommand{\entn}{\end{notation}}

\newcommand{\be}{\begin{enumerate}}
\newcommand{\ee}{\end{enumerate}}
\newcommand{\bce}{\begin{center}}
\newcommand{\ece}{\end{center}}
\newcommand{\bi}{\begin{itemize}}
\newcommand{\ei}{\end{itemize}}
\newcommand{\bt}{\begin{tabular}}
\newcommand{\et}{\end{tabular}}
\newcommand{\beq}{\begin{equation}}
\newcommand{\eeq}{\end{equation}}
\newcommand{\ba}{\begin{array}} 
\newcommand{\ea}{\end{array}}
\newcommand {\bea} {\begin{eqnarray}}
\newcommand {\eea} {\end {eqnarray}}
\newcommand {\bua} {\begin{eqnarray*}}
\newcommand {\eua} {\end {eqnarray*}}

\newcommand{\Ra}{\Rightarrow}

\newcommand{\se}{\subseteq}

\newcommand{\Lra}{\Leftrightarrow}

\newcommand{\uae}{The following are equivalent}

\newcommand{\ds}{\displaystyle}

\def\R{{\mathbb R}}
\def\N{{\mathbb N}}
\def\Z{{\mathbb Z}}

\def\R{{\mathbb R}}

\newcommand{\vp}{\varphi}
\newcommand{\eps}{\varepsilon}

\newcommand{\limn}{\ds\lim_{n\to\infty}}
\newcommand{\limk}{\ds\lim_{k\to\infty}}

\newcommand{\limp}{\ds\lim_{p\to\infty}}
\newcommand{\lsupn}{\ds \limsup_{n\to\infty}}

\newcommand{\lsupk}{\ds \limsup_{k\to\infty}}

\newcounter{ct}

\newcommand{\lambdaxy}{(1-\lambda)x\oplus\lambda y}

\title{Effective metastability of Halpern iterates in CAT(0) spaces}

\author{ U. Kohlenbach${}^1$, L. Leu\c stean$^{2}$\\[0.2cm]
\footnotesize ${}^1$ Department of Mathematics, Technische Universit\" at Darmstadt,\\
\footnotesize Schlossgartenstrasse 7, 64289 Darmstadt, Germany\\[0.1cm]
\footnotesize${}^2$ Simion Stoilow Institute of Mathematics of the Romanian Academy, \\
\footnotesize Research unit 5, P. O. Box 1-764, RO-014700 Bucharest, Romania\\[0.1cm]
\footnotesize E-mails: kohlenbach@mathematik.tu-darmstadt.de, Laurentiu.Leustean@imar.ro.
}

\begin{document}
\date{} % Sep.12, 2011
\maketitle

\begin{abstract}
This paper provides an effective uniform rate of metastability (in the
sense of Tao) on the strong convergence of Halpern iterations of
nonexpansive mappings in CAT(0) spaces. The extraction of this rate from
an ineffective proof due to Saejung is an instance of the general proof
mining program which uses tools from mathematical logic to uncover hidden
computational content from proofs. This methodology is applied here for
the first time to a proof that uses Banach limits and hence makes a
substantial reference to the axiom of choice.\\


\end{abstract}

\section{Introduction} 

This paper applies techniques from mathematical logic to extract an 
explicit uniform rate of metastability (in the sense of Tao \cite{Tao07,Tao08}) from a recent proof due to Saejung \cite{Sae10} of a strong convergence theorem for 
Halpern iterations in the context of CAT$(0)$ spaces. The theorem in 
question has been established originally in the context of Hilbert spaces 
by Wittmann in the important paper \cite{Wit92} and can there be viewed 
as a strong nonlinear generalization of the classical von Neumann mean ergodic 
theorem. 
Indeed, Wittmann's theorem says that under suitable 
conditions on a sequence of scalars $(\lambda_n)$ in $[0,1]$, including 
the case $\ds \lambda_n:=\frac{1}{n+1}$, the so-called Halpern iteration 
\[ x_0:=x, \ \ x_{n+1}:=\lambda_{n+1} x+(1-\lambda_{n+1})Tx_n \] 
of a nonexpansive selfmapping $T:C\to C$ of a bounded closed and convex subset 
$C\se X$ strongly converges to a fixed point of $T.$ If $T$ is, moreover, 
linear and $\ds \lambda_n:=\frac{1}{n+1},$ then $x_n$ coincides with the ergodic 
average $\ds \frac{1}{n+1}\sum^{n}_{i=0} T^ix$ from the mean ergodic 
theorem. 

Since Wittmann's theorem does not refer to any linearity but only to 
a convexity structure of the underlying space $X$ (in order to make sense 
of the Halpern iteration) it can be formulated in the context of hyperbolic 
spaces and was established by Saejung \cite{Sae10} for the important 
subclass of CAT(0) spaces which play the analogous role in the context of 
hyperbolic spaces as the Hilbert spaces do among all Banach spaces. 

As shown in \cite{AviGerTow10}, even for the (linear) mean 
ergodic theorem, there, in general, is no computable rate of convergence 
for $(x_n).$ The next best thing to achieve, therefore, is a 
rate of metastability, i.e. a bound $\Phi(k,g)$ such that 
\[ (1) \ \forall k\in \N\,\forall g:\N\to\N\,\exists n\le \Phi(k,g) \,
\forall i,j\in [n,n+g(n)]\ \big( \| x_i-x_j\| \le 2^{-k}\big). \]

There are general logical metatheorems due to the first author 
\cite{Koh05} and Gerhardy and the first author \cite{GerKoh08} 
that guarantee the extractability of computable and highly uniform such 
bounds $\Phi(k,g)$ from large classes of (even highly ineffective) proofs. 
Moreover, these bounds have a restricted complexity depending on the 
principles that are used in the proof rather than merely being computable 
(see \cite{Koh08-book} for a comprehensive treatment).

A rate of metastability is an instance of the concept of no-counterexample 
interpretation that was introduced in the context of mathematical logic 
by Kreisel in the 50's \cite{Kre51,Kre52}: 
as $g$ may be viewed as an attempt to refute the 
Cauchy property of $(x_n)$, the functional $\Phi(k,g)$ in $(1)$ 
provides a bound 
on a counterexample $n$ to such a refutation.  Note that since $g$ may be 
an arbitrary number theoretic function, the seemingly weaker form  
\[ (2) \ \forall k\in \N\,\forall g:\N\to\N\,\exists n\in\N \,
\forall i,j\in [n,n+g(n)]\ \big( \| x_i-x_j\| \le 2^{-k}\big) \] 
of the Cauchy property actually implies back the full Cauchy property, 
though only ineffectively so. Because of the latter point, the existence 
of an effective bound on $(2)$ does not contradict the aforementioned fact 
that there is no effective Cauchy rate for $(x_n)$ available. 

By the uniformity of the bound 
$\Phi$ we refer to the fact that it is independent of 
the operator $T,$ the point $x\in C$ as well as of $C$ and $X$ but only 
depends -- in addition to $k$ and $g$ -- on a bound on the diameter of $C$ as 
well as -- in the case of general $(\lambda_n)$ -- certain moduli on 
$(\lambda_n).$ 

Based on the aforementioned logical metatheorems, 
\cite{AviGerTow10} extracted the first explicit such uniform 
bound $\Phi$ for the mean ergodic theorem from its usual textbook proof. 
Subsequently, in \cite{KohLeu09} the current authors extracted 
such bound for the more general class of uniformly convex Banach spaces from 
a proof due to G. Birkhoff. That bound -- when specialized to the Hilbert 
space setting -- even turned out to be numerically better than the one 
from \cite{AviGerTow10}. 

In \cite{Koh11}, the first author extracted -- making use of a rate of 
asymptotic regularity due to the second author \cite{Leu07} --
a rate of metastability of 
similar complexity for Wittmann's nonlinear ergodic theorem (in the Hilbert 
case). Wittmann's proof is based on weak compactness which, though covered 
by the existing proof mining machinery, in general can cause bounds of 
extremely poor quality. In the case at hand that could be avoided as 
during the logical extraction procedure the use of weak compactness 
turned out to be eliminable.

In the present paper, we extract a rate of metastability from 
Saejung's \cite{Sae10} 
generalization of Wittmann's theorem to the CAT$(0)$-setting. 
In addition to the interest of this specific result, our paper is of 
broader relevance in the proof mining program as it opens up new 
frontiers for its applicability namely to proofs that prima facie use some 
substantial amount of the axiom of choice. This stems from the use 
of Banach limits made in \cite{Sae10}. The existence of Banach limits 
is either proved by applying the Hahn-Banach theorem to $l^{\infty}$ 
which due to the nonseparability of that space needs the axiom of choice, 
or via ultralimits which, again, needs choice. While weak compactness as used 
in Wittmann's proof at least was in principle covered by existing 
metatheorems mentioned above, this is not the case for Banach limits. 
Though it seems likely that these metatheorems can be extended to 
incorporate at least basic reasoning with Banach limits as we intend 
to discuss in a different paper, we take the route in this paper to 
show how to replace the use of Banach limits in the present 
proof by a direct arithmetical reasoning. 
As the way Banach limits are 
used in the proof at hand seems to be rather typical for other proofs 
in fixed point theory, our paper may also be seen as providing a 
blueprint for doing similar unwindings in those cases as well. Usually, 
a Banach limit is used to establish the almost convergence in the sense of 
Lorentz of some sequence $(a_n)$ of reals towards $a$ which -- together with 
$\lsupn (a_{n+1}-a_n)\le 0$ -- in turn implies that $\lsupn a_n\le a.$ 
This line of reasoning goes back to Lorentz' classical paper 
\cite{Lor48} whose relevance in nonlinear ergodic theory was first 
realized by Reich \cite{Rei78}. In \cite{ShiTak97}, Banach 
limits are used in this way to establish Wittmann's theorem for uniformly 
G\^{a}teaux differentiable Banach spaces (under suitable conditions on $C$). 
This paper has subsequently been analyzed using the method developed 
in this paper in \cite{KohLeu12}.
Other relavent papers using Banach limits in the context of nonlinear 
ergodic theory are \cite{BruRei:81,ReichWallwater,KopeckaReich}.

As an intermediate step in proving our main results we also obtain in Section \ref{effective-as-reg} (essentially due to the second author in \cite{Leu09-Habil}) a uniform effective rate of asymptotic regularity. i.e. a rate of convergence of $(d(x_n,Tx_n))$ towards $0$, which holds in general W-hyperbolic spaces. As this bound, in particular, does not depend on $x$ and $T$, it provides a quantitative version of the main result in \cite{AoyEshTak07} (see their `Theorem 3.3').
 
\section{Preliminaries}

We shall consider hyperbolic spaces as introduced by the first author \cite{Koh05}. In order to distinguish them from Gromov hyperbolic spaces or from other notions of hyperbolic space  that can be found in the metric fixed point theory literature (see for example \cite{Kir82,GoeKir83,ReiSha90}), we shall call them W-hyperbolic spaces. 

A \defnterm{$W$-hyperbolic space}  $(X,d,W)$ is a metric space $(X,d)$ together with a mapping $W:X\times X\times [0,1]\to X$ satisfying 
\begin{eqnarray*}
(W1) & d(z,W(x,y,\lambda))\le (1-\lambda)d(z,x)+\lambda d(z,y),\\
(W2) & d(W(x,y,\lambda),W(x,y,\tilde{\lambda}))=|\lambda-\tilde{\lambda}|\cdot 
d(x,y),\\
(W3) & W(x,y,\lambda)=W(y,x,1-\lambda),\\
(W4) & \,\,\,d(W(x,z,\lambda),W(y,w,\lambda)) \le (1-\lambda)d(x,y)+\lambda
d(z,w).
\end {eqnarray*}

The convexity mapping $W$ was first considered by Takahashi in \cite{Tak70}, where a triple $(X,d,W)$ satisfying $(W1)$ is called a \defnterm{convex metric space}. We refer to \cite[p. 384-387]{Koh08-book} for a detailed discussion.

The class of $W$-hyperbolic spaces includes normed spaces and convex subsets thereof, the Hilbert ball (see \cite{GoeRei84} for a book treatment)  as well as 
CAT(0) spaces \cite{BriHae99}.

If $x,y\in X$ and $\lambda\in[0,1]$, then we use the notation $(1-\lambda)x\oplus \lambda y$ for $W(x,y,\lambda)$. It is easy to see that for all $x,y\in X$ and  $\lambda\in[0,1]$,
\beq
d(x,\lambdaxy)=\lambda d(x,y)\quad \text{~and~}\quad  d(y,\lambdaxy)=(1-\lambda)d(x,y). \label{prop-xylambda}
\eeq
Furthermore, $1x\oplus 0y=x,\,0x\oplus 1y=y$ and $(1-\lambda)x\oplus \lambda x=\lambda x\oplus (1-\lambda)x=x$.

For all $x,y\in X$, we shall denote by $[x,y]$ the set $\{(1-\lambda)x\oplus \lambda y:\lambda\in[0,1]\}$. A subset $C\subseteq X$ is said to be \defnterm{convex} if $[x,y]\se C$ for all $x,y\in C$. A nice feature of our setting is that any convex subset is itself a $W$-hyperbolic space with the restriction of $d$ and $W$ to $C$.

Let us recall now some notions on geodesic spaces. Let $(X,d)$ be a metric space. A geodesic path, geodesic for short, in $X$ is a map $\gamma:[a,b]\to X$ which is distance-preserving, that is 
\beq 
d(\gamma(s),\gamma(t))=|s-t| \text{~~for all~~} s,t\in [a,b].
\eeq
A \defnterm{geodesic segment} in $X$ is the image of a geodesic $\gamma:[a,b]\to X$, the points $x:=\gamma(a)$ and $y:=\gamma(b)$ being the endpoints of the segment. We say that the geodesic  segment $\gamma([a,b])$ joins x and y. The metric space $(X,d)$ is said to be a \defnterm{(uniquely) geodesic space} if every two distinct points are joined by a (unique) geodesic segment. It is easy to see that any $W$-hyperbolic space is geodesic.

A \defnterm{CAT(0) space} is a geodesic space $(X,d)$ satisfying the so-called {\bf CN}-inequality of Bruhat-Tits \cite{BruTit72}: for all $x,y,z\in X$  and  $m\in X$ with $\ds d(x,m)=d(y,m)=\frac12 d(x,y)$,
\beq
d(z,m)^2\leq \frac12d(z,x)^2+\frac12d(z,y)^2-\frac14d(x,y)^2. \label{CN-ineq}
\eeq 

The fact that this definition of a CAT(0) space is equivalent to the usual definition using geodesic triangles is an exercise in \cite[p. 163]{BriHae99}. Complete CAT(0) spaces are often called Hadamard spaces. One can show that  CAT(0) spaces are uniquely geodesic and that a normed space is a CAT(0) space if and only if it is a pre-Hilbert space.

CAT(0) spaces can be defined also in terms of $W$-hyperbolic spaces.

\blem\label{char-CAT0}\cite[p. 386-388]{Koh08-book} 
Let $(X,d)$ be a metric space. \uae.
\be
\item \label{eq-CAT} $X$ is a {\rm CAT(0)} space.
\item \label{eq-W+CN} There exists a a convexity mapping $W$ such that $(X,d,W)$ is a $W$-hyperbolic space satisfying the {\bf CN} inequality (\ref{CN-ineq}).
\ee
\elem

The following property of  CAT(0) spaces will be very useful in the following. We refer to \cite[Lemma 2.5]{DhoPan08} for a proof.

\bprop
Let $(X,d)$ be a {\rm CAT(0)} space. Then  for all $x,y,z\in X$ and  $\lambda\in[0,1]$.
\beq
d^2((1-\lambda)x\oplus \lambda y, z) \leq (1-\lambda)d^2(x,z)+\lambda d^2(y,z)-\lambda(1-\lambda)d^2(x,y). \label{CAT0-ineq-t}
\eeq
\eprop

We recall now some terminology needed for our quantitative results. Let $(a_n)_{n\geq 1}$ be a sequence of real numbers and $a\in\R$. In the following $\N=\{0,1,2,\ldots\}$ and $\Z_+=\{1,2,\ldots\}$.

If the series $\ds \sum_{n=1}^\infty a_n$ is divergent, then a function $\gamma:\Z_+\to\Z_+$ is called a {\em rate of divergence} of the series if $\ds \sum_{k=1}^{\gamma(n)}a_k \geq n$ for all $n\in\Z_+$.

If $\ds\limn a_n=a$, then a function $\gamma:(0,\infty)\to\Z_+$ is said to be a {\em rate of convergence} of $(a_n)$ if 
\beq
\forall \eps>0\,\forall n\geq \gamma(\eps)\,\,\left(|a_n-a| \le \eps\right).\label{def-rate-conv}
\eeq

Assume that $(a_n)$ is Cauchy. Then 
\be
\item a mapping $\gamma:(0,\infty)\to\Z_+$ is called a {\em Cauchy modulus} of $(a_n)$ if 
\beq
\forall \eps>0\,\forall n\in\N\,\,\left(a_{\gamma(\eps)+n}-a_{\gamma(\eps)} \le \eps\right).\label{def-mod-Cauchy}
\eeq
\item a mapping $\Psi:(0,\infty)\times \N^\N\to\Z_+$ is called a \defnterm{rate of metastability} of  $(a_n)$ if 
\beq
\forall \eps>0\,\forall g:\N\to\N\,\exists N\le \Psi(\eps,g)\,\,\forall m,n\in[N,N+g(N)]\,\,(|a_n-a_m|\le \eps).\label{def-rate-meta}
\eeq
\ee

Finally, we say that $\ds \lsupn a_n\le 0$ with \defnterm{effective rate} $\theta:(0,\infty)\to\Z_+$ if
\beq
\forall \eps>0\,\forall n\ge \theta(\eps)\,\, (a_n\le \eps).\label{def-rate-limsup}
\eeq

\section{Halpern iterations}

Let $C$ be a convex subset of a normed space $X$ and  $T:C\to C$ nonexpansive. The  so-called {\em Halpern iteration} is defined as follows:
\beq
x_0:=x, \quad x_{n+1}:=\lambda_{n+1}u+(1-\lambda_{n+1})Tx_n,
\label{intro-def-Halpern-iterate}
\eeq
where $(\lambda_n)_{n\ge 1}$ is a sequence in $[0,1]$, $x\in C$ is the starting point and $u\in C$ is the anchor.

If $T$ is positively homogeneous (i.e. $T(tx)=tT(x)$ for all $t\geq 0$ and all $x\in C$), $\ds \lambda_n=\frac{1}{n+1}$ and $u=x$, then
\beq
x_n=\frac{1}{n+1}\,S_nx \label{xn-1-n+1-T-pos-homogeneous}, \quad\text{where}\quad S_0x=x, \,\,\,S_{n+1}x=x+T(S_nx).
\eeq
Furthermore, if $T$ is linear, then $\ds x_n=\frac{1}{n+1}\ds\sum_{i=0}^nT^ix$, so the Halpern iteration could be regarded as a nonlinear generalization of the usual Ces\`aro average. We refer to \cite{Wit91,LinWit91} for a a systematic study of the behavior of iterations given by (\ref{xn-1-n+1-T-pos-homogeneous}).

The following problem was formulated by Reich \cite{Rei83} (see also 
\cite{Rei74}) and it is still open in its full generality.

\begin{problem}\cite[Problem 6]{Rei83}\label{problem-Reich}\\
Let $X$ be a Banach space. Is there a sequence $(\lambda_n)$ such that whenever a weakly compact convex subset $C$ of $X$ possesses the fixed point property for nonexpansive mappings, then $(x_n)$ converges to a fixed point of $T$ for all $x\in C$ and all nonexpansive mappings $T:C\to C$ ?
\end{problem}

Different conditions on $(\lambda_n)$ were considered in the literature 
(see also \cite{Suz09} for even more conditions):
\[\ba{l}
(C1) \qquad \ds \lim_{n\to\infty} \lambda_n =0,\\
(C2) \qquad  \ds \sum_{n=1}^\infty|\lambda_{n+1}-\lambda_n| \text{ converges},\\
(C3) \qquad  \ds \sum_{n=1}^\infty \lambda_n=\infty,\\
(C4)  \qquad \ds \prod_{n=1}^\infty (1-\lambda_n)=0,
\ea\]
and, in the case $\lambda_n>0$ for all $n\ge 1$,
\[\ba{l}
(C5) \quad \ds\limn\frac{\lambda_n-\lambda_{n+1}}{\lambda_{n+1}^2}=0,\\
(C6) \quad  \ds\limn\frac{\lambda_n-\lambda_{n+1}}{\lambda_{n+1}}=0.
\ea
\]
For sequences $\lambda_n$ in $(0,1)$, conditions (C3) and (C4) are equivalent. 

Halpern \cite{Hal67} initiated the study in the Hilbert space setting of the convergence of a particular case of the scheme (\ref{intro-def-Halpern-iterate}). He proved that the sequence $(x_n)$, obtained by taking $u=0$ in (\ref{intro-def-Halpern-iterate}), converges to a fixed point of $T$ for $(\lambda_n)$ satisfying certain conditions, two of which are (C1) and (C3). P.-L. Lions \cite{Lio77} improved Halpern's result by showing the convergence of the general $(x_n)$ if $(\lambda_n)$ satisfies (C1), (C3) and (C5). However, both Halpern's and Lions' conditions exclude the natural choice $\ds\lambda_n=\frac{1}{n+1}$.

This was overcome by Wittmann \cite{Wit92}, who obtained the most important result on the convergence of Halpern iterations in Hilbert spaces.

\begin{theorem}\label{wittmann-thm}\cite{Wit92}
Let $C$  be a closed convex subset of a Hilbert space $X$ and $T:C\to C$ a nonexpansive mapping such that the set $Fix(T)$ of fixed points of $T$ is nonempty. Assume that $(\lambda_n)$ satisfies  (C1), (C2) and (C3).
Then for any $x\in C$, the Halpern iteration $(x_n)$ converges to the projection $Px$ of $x$ on $Fix(T)$.
\end{theorem}

All the partial answers to Reich's problem require that the sequence $(\lambda_n)$ satisfies (C1) and (C3).  Halpern \cite{Hal67} showed in fact that  conditions (C1) and (C3) are necessary in the sense that if, for every closed convex subset $C$ of a Hilbert space $X$ and every nonexpansive mappings $T:C\to C$ such that $Fix(T)\ne\emptyset$, the Halpern iteration $(x_n)$ converges to a fixed point of $T$, then $(\lambda_n)$ must satisfy (C1) and (C3). 
That (C1) and (C3) alone are not sufficient to guarantee the convergence of 
$(x_n)$ was shown in \cite{Suz09}. 
Recently, Chidume and Chidume \cite{ChiChi06} and Suzuki \cite{Suz07} proved that if  the nonexpansive mapping $T$ in (\ref{intro-def-Halpern-iterate}) is averaged, then (C1) and (C3) suffice for obtaining the convergence of $(x_n)$.

Halpern obtained his result by applying a limit theorem for the resolvent, first shown by Browder \cite{Bro67}. This approach has the advantage that the result can be immediately generalized, once the limit theorem for the resolvent is generalized. This was done by Reich \cite{Rei80}.

\bthm \cite{Rei80} \label{Reich-theorem-Halpern}
Let $C$ be a closed convex subset of a uniformly smooth Banach space $X$, and let $T:C\to C$ be nonexpansive such that $Fix(T)\ne\emptyset$. For each $u\in C$ and $t\in (0,1)$, let $z_t^u$ denote the unique fixed point of the contraction mapping 
$$T_t(\cdot)=tu+(1-t)T(\cdot).$$
Then $\ds\lim_{t\to 0^+}z_t^u$ exists and is a fixed point of $T$.
\ethm
A similar result was obtained by Kirk \cite{Kir03} for CAT(0) spaces 
(for the Hilbert ball, which is an example of a CAT(0) space, this is 
already due to \cite{GoeRei84}). 
As a consequence of Theorem \ref{Reich-theorem-Halpern}, a partial positive answer to Problem \ref{problem-Reich} was obtained \cite{Rei80} for uniformly smooth Banach spaces and $\ds\lambda_n=\frac1{(n+1)^\alpha}$ with $0<\alpha < 1$. Furthermore, Reich \cite{Rei94} proved the strong convergence of $(x_n)$ in the setting of uniformly smooth Banach spaces that have a weakly sequentially continuous duality mapping for general $(\lambda_n)$ satisfying (C1), (C3) and being decreasing. Another partial answer in the case of uniformly smooth Banach spaces was obtained by Xu \cite{Xu02} for $(\lambda_n)$ satisfying (C1), (C3) and (C6) (which is weaker than Lions' (C5)). In \cite{ShiTak97}, Shioji and Takahashi extended Wittmann's result  to Banach spaces with uniformly G\^ ateaux differentiable norm and  with the property that $\ds\lim_{t\to 0^+}z_t^u$ exists and is a fixed point of $T$.

\section{Main results}

Let $T:C\to C$ be a nonexpansive selfmapping of a convex subset $C$ of a  W-hyperbolic space
$(X,d,W)$. We can define the Halpern iteration in this setting too:
\beq
x_0:=x, \quad x_{n+1}:=\lambda_{n+1}u\oplus(1-\lambda_{n+1})Tx_n, \label{def-Halpern-iteration}
\eeq
where $x,u\in C$ and $(\lambda_n)_{n\ge 1}$ is a sequence in $[0,1]$. 

The following theorem generalizes Wittman's theorem to CAT(0) spaces and was obtained by Saejung \cite{Sae10} (as similar result for the Hilbert ball had 
already been proved in \cite{KopeckaReich}).
 
\bthm\label{thm-Saejung}
Let $C$  be a closed convex subset of a complete {\rm CAT(0)} 
space $X$ and $T:C\to C$ a nonexpansive mapping such that the set $Fix(T)$ of fixed points of $T$ is nonempty. Assume that $(\lambda_n)$ satisfies  (C1), (C2) and (C3). Then for any $u,x\in C$, the iteration $(x_n)$ converges to the projection $Pu$ of $u$ on $Fix(T)$.
\ethm
By \cite[Theorem 18]{Kir03}, $Fix(T)\ne\emptyset$ is guaranteed to hold 
if $C$ is bounded. In this paper we only consider this case and our 
bounds will depend on an upper bound $M$ on the diameter $d_C$ of $C.$ 
However, similar 
to \cite{Koh11}, it is 
not hard to adopt our bounds to the case where the condition $M\ge d_C$ 
is being replaced by $M\ge d(u,p),d(x,p)$ 
for some fixed point $p\in C$ of $T.$ 
\\[1mm]
The main results of the paper are  effective versions of Theorem \ref{thm-Saejung}, obtained by applying proof mining techniques to Saejung`s proof. As this proof is essentially ineffective and -- as we discussed in the introduction -- a computable rate of convergence does not exist, while an effective and
highly uniform rate of metastability (depending only on the input data displayed in Theorems \ref{Halpern-rate}, \ref{Halpern-rate-1}) is guaranteed to exist (via our elimination of Banach limits from the proof) by 
\cite[Theorem 3.7.3]{Koh05} 
(note that the conditions on $\alpha,\beta,\theta$ as well as $T$ are 
all purely universal while the conclusion 
$\exists N \forall m,n\in [N,N+g(N)]\ (d(x_n,x_m)<\eps)$ 
can be written as a purely existential formula and that quantification 
over all $(\lambda_n)$ in $[0,1]$ can be represented as $\forall y\le s$ 
for some simple function $s:\N^2\to\N$).

\bthm \label{Halpern-rate}
Assume that $X$ is a complete {\rm CAT(0)} 
space, $C\se X$ is a  closed bounded convex subset with diameter $d_C$ and  $T:C\to C$ is nonexpansive. Let $(\lambda_n)$ satisfy (C1), (C2) and (C3). 

Then the Halpern iteration $(x_n)$ is Cauchy.\\
Furthermore, let $\alpha$ be a rate of convergence of $(\lambda_n)$, $\beta$ be  a Cauchy modulus of $s_n:=\ds\sum_{i=1}^n|\lambda_{i+1}-\lambda_i|$ and $\theta$ be a rate of divergence
of $\ds\sum_{n=1}^\infty \lambda_{n+1}$.

Then for all $\eps\in(0,2)$ and  $g:\N\to\N$,
\[
\exists N\le \Sigma(\eps,g,M,\theta,\alpha,\beta)\,\,\forall m,n\in[N,N+g(N)]\,\,(d(x_n,x_m)\le \eps),
\]
where 
\beq
\Sigma(\eps,g,M,\theta,\alpha,\beta)= \theta^+\left(\Gamma-1+\left\lceil\ln\left(\frac{12 M^2}{\eps^2}\right)\right\rceil\right)+1
\eeq
with $M\in\Z_+$ such that $M\ge d_C$,
\bua
\eps_0 =\frac{\eps^2}{24(M+1)^2}, & \ds \Gamma =\max \left\{\chi^*_k(\eps^2/12) \mid  \left\lceil\frac{1}{\eps_0}\right\rceil\le k\le \widetilde{f^*}^{(\lceil M^2/\eps_0^2\rceil)}(0)+\left\lceil\frac{1}{\eps_0}\right\rceil\right\},
\eua
\bua
\chi^*_k(\eps)= \tilde{\Phi}\left(\frac{\eps}{4M(\tilde{P}_k\left(\eps/2\right)+1)}\right)+\tilde{P}_k\left(\eps/2\right), & \ds \tilde{P}_k\left(\eps\right)= \left\lceil\frac{12M^2(k+1)}{\eps}\Phi\left(\ds\frac{\eps}{12M(k+1)}\right)\right\rceil,\\
\eua
\bua
\tilde{\Phi}(\eps,M,\theta,\beta)=\theta\left(\beta\left(\frac\eps {4M}\right)+1+\left\lceil\ln\left(\frac{2M}\eps\right)\right\rceil\right)+1, \\
\Phi(\eps,M,\theta,\alpha,\beta) = \max\left\{\tilde{\Phi}\left(\frac{\eps}2,M,\theta,\beta\right),\alpha\left(\frac\eps {4M}\right)\right\}, 
\eua
\bua
\ds \Delta^*_k(\eps,g)=\frac{\eps}{3g_{\eps,k}\left(\Theta_k(\eps)-\chi^*_k(\eps/3)\right)}, & \ds \Theta_k(\eps)=  \theta\left(\chi^*_k\left(\frac{\eps}3\right)-1+\left\lceil\ln\left(\frac{3 M^2}{\eps}\right)\right\rceil\right)+1, \\
g_{\eps,k}(n)=n+g\left(n+\chi^*_k\left(\frac{\eps}3\right)\right), & \theta^+(n)= \max\{\theta(i)\mid i\le n\},
\eua
\bua
f(k) = \max\left\{\left\lceil\frac{M^2}{\Delta^*_k(\eps^2/4,g)}\right\rceil, k\right\}-k,& 
\ds f^*(k) = f\left(k+\left\lceil\frac{1}{\eps_0}\right\rceil\right)+\left\lceil\frac{1}{\eps_0}\right\rceil, & \ds \widetilde{f^*}(k)= k+f^*(k).
\eua
\ethm 
\begin{proof}
See Section \ref{proof-Halpern-rate}.
\end{proof}

A similar result can be obtained by assuming that $(\lambda_n)$ satisfies (C1), (C2) and (C4) with corresponding rates.

\bthm\label{Halpern-rate-1}
Assume that $X$ is a complete {\rm CAT(0)} space, $C\se X$ is a closed bounded convex subset with diameter $d_C$ and  $T:C\to C$ is nonexpansive. Let $(\lambda_n)$ satisfy (C1), (C2), (C4) and $\lambda_n \in (0,1)$ for all $n\ge 2$.

Then the Halpern iteration $(x_n)$ is Cauchy.\\
Furthermore, if $\alpha$ is a rate of convergence of $(\lambda_n)$, $\beta$ is  a Cauchy modulus of $s_n:=\ds\sum_{i=1}^n|\lambda_{i+1}-\lambda_i|$ and $\theta$ is a rate of convergence of $\ds \prod_{n=1}^\infty (1-\lambda_{n+1})$ towards $0$, then for all $\eps\in (0,2)$ and  $g:\N\to\N$,
\bua
\exists N\le \Sigma(\eps,g,M,\theta,\alpha,\beta,(\lambda_n))\,\,\forall m,n\in[N,N+g(N)]\,\,(d(x_n,x_m)\le \eps),
\eua
where 
\beq
\Sigma(\eps,g,M,\theta,\alpha,\beta,(\lambda_n)):= \\max\left\{\Theta_k
(\eps^2/4) \mid \left\lceil\frac{1}{\eps_0}\right\rceil\le k\le \widetilde{f^*}^{(\lceil M^2/\eps_0^2\rceil)}(0)+\left\lceil\frac{1}{\eps_0}\right\rceil\right\},
\eeq
with $M\in\Z_+$ such that $M\ge d_C$,
\[ 0<D\le \prod_{n=1}^{\beta(\eps/4M)}(1-\lambda_{n+1}), \]
\bua
\tilde{\Phi}(\eps,M,\theta,\beta,D)&=& \theta\left(\frac{D\eps}{2M}\right)+1,\\ 
\Phi(\eps,M,\theta,\alpha,\beta,D) & = & \max\left\{\theta\left(\frac{D\eps}{4M}\right)+1,\alpha\left(\frac\eps {4M}\right)\right\}, \\
\Theta_k(\eps) &=& \theta\left(\frac{D_k\eps}{3M^2}\right)+1,\\
0<D_k&\le & \prod_{n=1}^{\chi_k^*(\eps/3)-1} (1-\lambda_{n+1}), 
\eua
and the other constants and functionals being defined as in Theorem \ref{Halpern-rate}.
\ethm
\begin{proof}
We use Proposition \ref{Halpern-rate-as-reg-hyperbolic-1},  Lemma \ref{quant-Aoyama-all-meta-rate-1} and follow the same line as in the proof of Theorem \ref{Halpern-rate}.
\end{proof}

One can modify Theorems \ref{Halpern-rate}, \ref{Halpern-rate-1} so that only metastable versions of $\alpha,\beta$ and $\theta$ are needed. However, we refrain from doing so as the result would be rather unreadable and  in the practical cases at hand -- such as $\ds \lambda_n=\frac 1{n+1}$ -- full rates $\alpha,\beta,\theta$ are easy to compute.

\bcor
Assume that $\ds\lambda_n=\frac 1{n+1}$ for all $n\geq 1$. Then for all $\eps\in (0,1)$ and  $g:\N\to\N$,
\bua
\exists N\le \Sigma(\eps,g,M)\,\,\forall m,n\in[N,N+g(N)]\,\,(d(x_n,x_m)\le \eps),
\eua
where 
\beq
\Sigma(\eps,g,M)= \left\lceil \frac{12M^2(\chi_L^*(\eps^2/12)+1)}{\eps^2}
\right\rceil-1
\eeq
with 
\bua
L& =& \widetilde{f^*}^{(\lceil M^2/\eps_0^2\rceil)}(0)+\left\lceil\frac{1}{\eps_0}\right\rceil, \\
\tilde{P}_k(\eps)&=& \left\lceil \frac{12M^2(k+1)}{\eps}\cdot\left(\left\lceil 
\frac{48M(k+1)}{\eps}+\frac{2304M^4(k+1)^2}{\eps^2}
\right\rceil -1\right)\right\rceil,\\
\chi_k^*(\eps) &=& \left\lceil\frac{8M^2(\tilde{P}_k\left(\eps/2\right)+1)}{\eps}+\frac{128M^4(\tilde{P}_k\left(\eps/2\right)+1)^2}{\eps^2}\right\rceil -1+\tilde{P}_k\left(\eps/2\right),\\
\Theta_k(\eps) &=& \left\lceil \frac{3M^2(\chi_k^*(\eps/3)+1)}{\eps}\right\rceil-1, 
\eua
while the  other constants and functionals are defined as in Theorem \ref{Halpern-rate}.
\ecor
\begin{proof}
Since $\ds \prod_{k=1}^n\left(1-\frac{1}{k+2}\right)=\frac{2}{n+2}$, we get that $\ds \theta(\eps):=\left\lceil\frac{2}{\eps}\right\rceil-2$ is a rate of convergence of $\ds \prod_{n=1}^\infty\left(1-\frac{1}{n+2}\right)$ towards $0$. Furthermore, 
we can take $\ds D_k:=\frac{2}{\chi_k^*(\eps/3)+1}$ in Theorem \ref{Halpern-rate-1}  and -- using 
Corollary \ref{Halpern-rate-as-reg-hyperbolic-1n} -- $\Phi:=\Psi, \tilde{\Phi}
:=\tilde{\Psi}$ from that corollary. We then get $P_k(\eps),\, \chi_k^*(\eps)$ as above and 
\bua
\Theta_k(\eps) &=& \theta\left(\frac{D_k\eps}{3M^2}\right)+1=  \left\lceil \frac{3M^2(\chi_k^*(\eps/3)+1)}{\eps}\right\rceil-1.
\eua
The claim now follows by (the proof of) 
Theorem \ref{Halpern-rate-1} using that $\chi_k^*$ increases with $k$.
\end{proof}

Despite its superficially quite different look, the bound in Corollary 4.4 has an overall similar structure as the bound extracted for the Hilbert space case in \cite{Koh11}: the bound results from applying a certain function $\Theta_k(\eps)$ to a number $k:=L$ which is the result of an iteration of a function $\tilde{f^*}$ (starting at some arbitrary value, e.g. $0$), where $\tilde{f^*}(k)$ is -- disregarding many details -- something close to $\Theta_k(\eps)+g(\Theta_k(\eps)).$  This is also the structure of the bound in \cite[Theorem 3.3]{Koh11} (where
$\Delta^*$ plays the role of $\tilde{f^*}$). Note that the number of 
iterations essentially is $M^6/\eps^4$ while it was roughly $M^4/\eps^4$ in  
the bound in 
\cite[Theorem 3.3]{Koh11}. The main difference, though, is that 
now $\Theta_k$ is significantly more involved compared to \cite{Koh11} 
(most of its terms stemming from the remains of the original 
Banach-limit argument).
\begin{remark} 
\be
\item 
By replacing $(X,d)$ by $(X,d_M)$ with $\ds d_M(x,y):=\frac{1}{M}d(x,y)$ one 
can always arrange that $1\ge d_C$ and then apply the above bounds for 
$1$ instead of $M$ but with $\eps/M$ instead of $\eps$ to compensate for 
this rescaling. One then gets a bound in which $\eps$ and $M$ only occur in 
the form $\eps/M$ and the number of iterations is (essentially) $M^4/\eps^4.$ 
However, in doing so $M$ would enter the bound at many unnecessary places 
as well.
\item 
The assumption on the completeness of $X$ and the closedness of $C$ 
facilitates the proofs but 
is not necessary in the above results. If the results would fail for an 
incomplete $X$ then it is easy to show that they would fail already for 
the metric completion $\widehat{X}$ of $X$ and the closure $\overline{C}$ 
of $C$ in $\widehat{X}$ (since $T$ extends to a 
nonexpansive operator $\widehat{T}:\overline{C}\to\overline{C}$). 
Alternatively, one could use directly appropriate 
approximate fixed points rather than fixed points in 
the applications of Banach's fixed point theorem in section 
\ref{section-Banach} below. 
\item 
Subsequently, our results have been further generalized in 
\cite{Schade} to the case of unbounded $C$ provided that $T$ 
possesses a fixed point $p.$ Then the above bounds hold with 
$M\ge diam(C)$ being replaced by $M\ge 4\max\{ d(u,x),d(u,p)\}.$ 
In  \cite{Schade} our method is also adapted to obtain similar 
bounds for more general schemes of so-called modified Halpern 
iterations.
\ee
\end{remark}

\section{Quantitative lemmas on sequences of real numbers}

The following lemma about sequences of real numbers was proved in \cite{AoyKimTakToy07}.

\blem\label{Aoyama-all-seq-reals}
Let $(s_n)$ be a sequence of nonnegative real numbers, $(\alpha_n)$ be a sequence of real numbers in $[0,1]$
 with $\ds \sum_{n=1}^\infty \alpha_n=\infty$,  and $(t_n)$ be a sequence of real numbers with $\ds \lsupn t_n\le 0$. Suppose that 
\[
s_{n+1}\le (1-\alpha_n)s_n+\alpha_nt_n \quad\text{for all }n\ge 1.
\]
Then $\ds \limn s_n=0$.
\elem

We prove now quantitative versions of Lemma \ref{Aoyama-all-seq-reals}, which also allow for an error term $\Delta$.

\blem\label{quant-Aoyama-all-meta-rate-0}
Let $\eps \in (0,2)$, 
$g:\N\to\N$, $ M\in\Z_+$, $\theta:\Z_+\to\Z_+$ and $\psi:(0,\infty)\to\Z_+$. Define
\bea
\Theta:=\Theta(\eps, M,\theta,\psi) & = & \theta\left(\psi\left(\frac{\eps}3\right)-1+\left\lceil\ln\left(\frac{3 M}{\eps}\right)
\right\rceil\right)+1, \label{def-Theta-quant-lemma-0}\\
\Delta:=\Delta(\eps,g, M,\theta,\psi) &= &  \frac{\eps}{3g_\eps(\Theta-\psi(\eps/3))},
\label{def-Delta-quant-lemma-0}
\eea
where  $g_\eps(n)=n+g(n+\psi(\eps/3))$.\\
Assume that $(\alpha_n)$ is a sequence in $[0,1]$ such that $\ds \sum_{n=1}^\infty \alpha_n=\infty$ with rate of divergence $\theta$. Let $(t_n)$ be a sequence of real numbers satisfying 
\beq
\forall n\ge \psi(\eps/3)\,\, (t_n\le \eps/3).\label{ineq(30)}
\eeq
Let  $(s_n)$ be a bounded sequence with upper bound $ M$ satisfying
\beq
s_{n+1}\le (1-\alpha_n)s_n+\alpha_nt_n+\Delta \quad\quad\text{for all }n\ge 1.\label{sn-recurrence-0}
\eeq
Then 
\[
\forall n\in[\Theta,\Theta+g(\Theta)] \,\,(s_n\le \eps).
\]
\elem
\begin{proof} By induction on $m$ one shows that for all $n\ge \psi(\eps/3)$ and  $m\ge 1$,
\beq
s_{n+m}\leq \left[\prod_{j=n}^{n+m-1}(1-\alpha_j)\right]s_n+\left[1-\prod_{j=n}^{n+m-1}(1-\alpha_j)\right]\frac{\eps}3+m\Delta.
\label{ineq-prod-s-n-Delta-0}
\eeq
$m=1$: By (\ref{sn-recurrence-0}) and (\ref{ineq(30)}), we have that 
\bua
s_{n+1}&\le & (1-\alpha_n)s_n+\alpha_nt_n+\Delta \le (1-\alpha_n)s_n+\alpha_n\frac{\eps}3+\Delta\\
&=& (1-\alpha_n)s_n+(1-(1-\alpha_n))\frac{\eps}3+\Delta.
\eua
$m\Ra m+1$: We have that
\bua
s_{n+m+1} &\le & (1-\alpha_{n+m})s_{n+m}+\alpha_{n+m}t_{n+m}+\Delta\\
&\le & (1-\alpha_{n+m})\left[\prod_{j=n}^{n+m-1}(1-\alpha_j)\right]s_n+(1-\alpha_{n+m})\left[1-\prod_{j=n}^{n+m-1}(1-\alpha_j)\right]\frac{\eps}3+\\
&& +(1-\alpha_{n+m})m\Delta+\alpha_{n+m}t_{n+m}+\Delta \quad\quad\text{by the induction hypothesis}\\
&\le & \left[\prod_{j=n}^{n+m}(1-\alpha_j)\right]s_n+(1-\alpha_{n+m})\left[1-\prod_{j=n}^{n+m-1}(1-\alpha_j)\right]\frac{\eps}3+\alpha_{n+m}\frac{\eps}3+(m+1)\Delta\\
&=& \left[\prod_{j=n}^{n+m}(1-\alpha_j)\right]s_n+\left[1-\alpha_{n+m}-\prod_{j=n}^{n+m}(1-\alpha_j)+\alpha_{n+m}\right]\frac{\eps}3+(m+1)\Delta\\
&= & \left[\prod_{j=n}^{n+m}(1-\alpha_j)\right]s_n+\left[1-\prod_{j=n}^{n+m}(1-\alpha_j)\right]\frac{\eps}3+(m+1)\Delta.
\eua
Using the fact that $1-x\leq \exp(-x)$ for all $x\in[0,\infty)$, we get that 
$$ \ds \prod_{j=n}^{n+m-1}(1-\alpha_j)\le \prod_{j=n}^{n+m-1} \exp(-\alpha_j)=\exp\left(-\sum_{j=n}^{n+m-1}\alpha_j\right),$$ 
hence 
\beq
s_{n+m} \le  \exp\left(-\sum_{j=n}^{n+m-1}\alpha_j\right)s_n+\frac{\eps}3+m\Delta
\le  \exp\left(-\sum_{j=n}^{n+m-1}\alpha_j\right)M+\frac{\eps}3+m\Delta \label{quant-lemma-useful-1-0}
\eeq
for all $n\ge \psi(\eps/3)$ and  $m\ge 1$. 

For simplicity, let us denote $\ds d_{m,n}:= M\exp\left(-\sum_{j=n}^{n+m-1}\alpha_j\right)$.
As in \cite{Leu07}, we get that
\bua
d_{m,n}\le \frac{\eps}3 & \Lra & \exp\left(-\sum_{j=n}^{n+m-1}\alpha_j\right)\leq \frac{\eps}{3 M} \,\,\Lra\,\,  -\sum_{j=n}^{n+m-1}\alpha_j\leq \ln\left(\frac{\eps}{3 M}\right)\\
&\Lra & \sum_{j=n}^{n+m-1}\alpha_j\geq -\ln\left(\frac{\eps}{3 M}\right)=\ln\left(\frac{3 M}{\eps}\right)\,\,\Lra \,\,\sum_{j=1}^{n+m-1}\alpha_j\geq \sum_{j=1}^{n-1}\alpha_j+\ln\left(\frac{3 M}{\eps}\right).
\eua
Let 
\beq
L:=\Theta-\psi(\eps/3)=\theta\left(\psi(\eps/3)-1+\left\lceil\ln\left(\frac{3 M}{\eps}\right)
\right\rceil\right)+1-\psi(\eps/3).
\eeq
Since $\theta$ is a rate of divergence of $\ds\sum_{n=1}^\infty \alpha_n$ and $\alpha_n\leq 1$, it is obvious that $\theta(n)\geq n$ for all $n\ge 1$, hence $L\ge 1$. For all $m\ge L$, we have that 
\bua
\sum_{j=1}^{\psi(\eps/3)+m-1}\alpha_j& \ge & \sum_{j=1}^{\psi(\eps/3)+L-1}\alpha_j\ge \psi(\eps/3)-1+\left\lceil\ln\left(\frac{3 M}{\eps}\right)\right\rceil\ge \sum_{j=1}^{\psi(\eps/3)-1}\alpha_j+\ln\left(\frac{3M}{\eps}\right),
\eua
hence 
\bua
d_{m,\psi(\eps/3)}\le \frac{\eps}{3}\quad\text{for all }m\ge L.
\eua
Apply now (\ref{quant-lemma-useful-1-0}) with $n:=\psi(\eps/3)$ to get that for all $m\ge L$,
\beq
s_{\psi(\eps/3)+m} \le \frac{2\eps}3+m\Delta. \label{useful-s-m-psi-0}
\eeq
Let $n\in[\Theta,\Theta+g(\Theta)]$. Then $$L\le n-\psi(\eps/3)\le \Theta+g(\Theta)-\psi(\eps/3)=L+g(L+\psi(\eps/3))=g_\eps(L),$$ hence we can apply (\ref{useful-s-m-psi-0}) with $m:=n-\psi(\eps/3)$ to get that
\[
s_n \le \frac{2\eps}3+g_\eps(L)\Delta=\eps.
\]
\end{proof}

It is well-known that for a sequence $(\alpha_n)$ 
in $(0,1)$ we have that $\ds \sum_{n=1}^\infty \alpha_n=\infty$ if and only if $\ds \prod_{n=1}^\infty (1-\alpha_n)=0$. This suggests a second quantitative version of Lemma \ref{Aoyama-all-seq-reals}, where, instead of a  rate of divergence for 
$\ds \sum_{n=1}^\infty \alpha_n$, we assume the existence of a rate of convergence of $\ds \prod_{n=1}^\infty (1-\alpha_n)$ towards $0$.

\blem\label{quant-Aoyama-all-meta-rate-1}
Let $\eps>0$, $g:\N\to\N$, $M\in\Z_+$, $D>0$ and $\theta,\psi:(0,\infty)\to\Z_+$. Define
\bea
\Theta:=\Theta(\eps, M,\theta,\psi,D) & = & \max\left\{\theta\left(\frac{D\eps}{3M}\right)+1,\psi\left(\frac{\eps}3\right)\right\} \label{def-Theta-quant-lemma-1}\\
\Delta:=\Delta(\eps,g, M,\theta,\psi,D) &= &  \frac{\eps}{3g_\eps(\Theta-\psi(\eps/3))},
\label{def-Delta-quant-lemma-1}
\eea
where $g_\eps(n)=n+g(n+\psi(\eps/3))$.\\
Assume that $(\alpha_n)$ is a sequence in $(0,1)$ such that $\ds \prod_{n=1}^\infty (1-\alpha_n)=0$ with rate of convergence $\theta$. Let $(t_n)$ be a sequence of real numbers satisfying 
\beq
\forall n\ge \psi(\eps/3)\,\, (t_n\le \eps/3).
\eeq
Assume furthermore that
\beq
D\le \prod_{n=1}^{\psi(\eps/3)-1} (1-\alpha_n). \label{extra-hyp-D-quant-lemma}
\eeq
Let  $(s_n)$ be a bounded sequence with upper bound $M$ satisfying
\beq
s_{n+1}\le (1-\alpha_n)s_n+\alpha_nt_n+\Delta \quad\quad\text{for all }n\ge 1.\label{sn-recurrence-1}
\eeq
Then 
\[
\forall n\in[\Theta,\Theta+g(\Theta)] \,\,(s_n\le \eps).
\]
\elem
\begin{proof}
We shall denote $P_n:=\ds \prod_{j=1}^n(1-\alpha_j)$ for all $n\ge 1$. By convention, $P_0=1$. We get as in the proof of Lemma \ref{quant-Aoyama-all-meta-rate-0} that for all $n\ge \psi(\eps/3)$ and  $m\ge 1$,
\beq
s_{n+m}\leq \left[\prod_{j=n}^{n+m-1}(1-\alpha_j)\right]s_n+\left[1-\prod_{j=n}^{n+m-1}(1-\alpha_j)\right]\frac{\eps}3+m\Delta.
\label{ineq-prod-s-n-Delta-1}
\eeq
Hence, for all $n\ge \psi(\eps/3)$ and   $m\ge 1$,
\bua
s_{n+m} &\le &  \left[\prod_{j=n}^{n+m-1}(1-\alpha_j)\right]s_{n}+\frac{\eps}3+m\Delta
\le  \left[\prod_{j=n}^{n+m-1}(1-\alpha_j)\right]M+\frac{\eps}3+m\Delta\\
&=& \frac{MP_{n+m-1}}{P_{n-1}}+\frac{\eps}3+m\Delta.
\eua
By taking $n:=\psi(\eps/3)$, we get that for all  $m\ge 1$,
\beq
s_{\psi(\eps/3)+m} \le  \frac{MP_{\psi(\eps/3)+m-1}}{P_{\psi(\eps/3)-1}}+\frac{\eps}3+m\Delta. \label{useful-s-m-psi-1}
\eeq
Define now 
\bea
L := \Theta-\psi(\eps/3)=\max\left\{\theta\left(\frac{D\eps}{3M}\right)+1-\psi(\eps/3),0\right\}
\eea
and take $n\in[\Theta,\Theta+g(\Theta)]$ arbitrary. Then $L \le n-\psi(\eps/3)\le g_\eps(L)$ and, applying (\ref{useful-s-m-psi-1}) with $m:=n-\psi(\eps/3)$, it follows that
\bua
s_n & \le & \frac{MP_{n-1}}{P_{\psi(\eps/3)-1}}+\frac{\eps}3+(n-\psi(\eps/3))\Delta
\le \frac{MP_{\Theta-1}}{P_{\psi(\eps/3)-1}}+\frac{\eps}3+g_\eps(L)\Delta\le 
\frac{M}{P_{\psi(\eps/3)-1}}\cdot\frac{D\eps}{3M} +\frac{2\eps}3,
\eua
as $\ds \Theta-1\ge \theta\left(\frac{D\eps}{3M}\right)$. By (\ref{extra-hyp-D-quant-lemma}), we get that $s_n\le \eps$.
\end{proof}

The above lemma turns out to be very useful to get better bounds in the case  $\ds \alpha_n=\frac{1}{n+1}$, as
$\ds \sum_{n=1}^\infty \frac{1}{n+1}$ has an exponential rate of divergence, while $\ds \prod_{n=1}^\infty \left(1-\frac{1}{n+1}\right)$ has a linear rate of convergence towards $0$.

\bcor\label{quant-Aoyama-1n}
Let $\eps\in(0,3)$, $g:\N\to\N$, $M\in\Z_+$, $\psi:(0,\infty)\to\Z_+$. Define
\bea
\Theta:=\Theta(\eps, M,\psi)=\left\lceil\frac{3M\psi(\eps/3)}{\eps}\right\rceil+1, \quad \Delta:=\Delta(\eps,g, M,\psi) =  \frac{\eps}{3g_\eps(\Theta-\psi(\eps/3))},
\eea
where $g_\eps(n)=n+g(n+\psi(\eps/3))$.\\
Assume that $(t_n)$ is a sequence of real numbers satisfying 
\beq
\forall n\ge \psi(\eps/3)\,\, (t_n\le \eps/3).
\eeq
Let  $(s_n)$ be a bounded sequence with upper bound $M$ satisfying
\beq
s_{n+1}\le \left(1-\frac{1}{n+1}\right)s_n+\frac{1}{n+1}t_n+\Delta \quad\quad\text{for all }n\ge 1.
\eeq
Then 
\beq
\forall n\in[\Theta,\Theta+g(\Theta)] \,\,(s_n\le \eps).
\eeq
\ecor
\begin{proof}
Remark that for all $n\ge 1$, we have that $\ds \prod_{k=1}^n\left(1-\frac{1}{k+1}\right)=\frac{1}{n+1}$, hence $\ds \theta(\eps):=\left\lceil\frac{1}{\eps}\right\rceil$ 
is a rate of convergence of $\ds \prod_{n=1}^\infty \left(1-\frac{1}{n+1}\right)$ towards $0$. Furthermore, we can take $\ds D:=\frac{1}{\psi(\eps/3)}$ in Lemma \ref{quant-Aoyama-all-meta-rate-1}. Since $\eps\in(0,3)$, we have that $\ds \frac{3M\psi(\eps/3)}{\eps}\ge \psi(\eps/3)$, hence $\ds\left\lceil\frac{3M\psi(\eps/3)}{\eps}\right\rceil+1> \psi(\eps/3)$.
\end{proof}

The proof of Lemmas \ref{quant-Aoyama-all-meta-rate-0}, \ref{quant-Aoyama-all-meta-rate-1} can actually be reformulated to give a full rate of convergence for $(s_n)$ provided that one does not have the error term $\Delta$ or that $\Delta$ can be made arbitrarily small while still 
keeping $\psi$ and (\ref{ineq(30)}) unchanged (note that $\Theta$ -- in contrast to $\Delta$ -- does not depend on $g$). This error term stems from the fact that we have to eliminate a use of an ineffective arithmetical comprehension hidden in forming the limit $z$ of a certain sequence of points $(z_{t_k})$ which is used in Saejung's proof to construct the sequence 
which plays the role of $(t_n)$ in the use of Lemma \ref{quant-Aoyama-all-meta-rate-0} or Lemma 
\ref{quant-Aoyama-all-meta-rate-1} (see \cite[(2.21)-(2.23)]{Sae10}). 
Instead of $z,$ we take $z_{t_k}$ where $k$ is 
sufficiently large so that $d(z_{t_j},z)< \eps$ for all $j\ge k.$ 
This error can be incorporated (also when switching from 
$z_{t_k}$ to $z_{t_j}$ for $j\ge k$) 
into the error already present in (\ref{ineq(30)}) with some $\psi_k$ depending 
on $k$ but it adds the error $\Delta_j:=M^2t_j$ (see (\ref{d2xn+1-zt-v3}) below compared to 
\cite[(2.21)]{Sae10}), which we provided for in (\ref{sn-recurrence-0}). The error $\Delta_j,$  
however, can be made arbitrarily small by increasing $j$ without changing $\psi_k$ in (\ref{ineq(30)}) (see the proof of the main Theorem \ref{Halpern-rate}).  
This would give us a rate of convergence in our Theorem \ref{Halpern-rate} provided that 
we had a Cauchy rate on $(z_{t_k}).$ However, we effectively only get a rate of metastability 
for this sequence (see Proposition \ref{quant-Browder-Halpern} and the discussion 
preceding this proposition). As a result, $k$ and in turn $\psi_k$ become dependent on the counterfunction $g.$ This has the consequence that now, via $\psi_k$, also $\Theta$ in our application of Lemma \ref{quant-Aoyama-all-meta-rate-0} (in the proof of Theorem \ref{Halpern-rate}) becomes dependent on $g.$ It is this issue which is responsible for the fact that we only get an effective rate of metastability in Theorem \ref{Halpern-rate} 
(rather than a Cauchy rate), which -- as discussed in the introduction -- in fact is 
best possible.    

The following quantitative lemma is  the main ingredient in getting effective rates of asymptotic regularity for the Halpern iteration.

\begin{lemma}\label{quant-liu} 
Let $(\lambda_n)_{n\ge 1}$ be a sequence in $[0,1]$ and $(a_n)_{n\geq 1},(b_n)_{n\geq 1}$ be sequences in $\R_+$  such that for all $\ds n\ge 1$,
\beq
a_{n+1}\leq (1-\lambda_{n+1}) a_n + b_n.
\eeq
Assume that $\ds \sum_{n=1}^\infty b_n$ is convergent and $\gamma$ is a Cauchy modulus of $\ds s_n:=\ds\sum_{i=1}^nb_i$. 
\be
\item\label{quant-liu-1} If  $\ds \sum_{n=1}^\infty \lambda_{n+1}$ is divergent with rate of divergence $\theta$, then
\[
 \forall \eps\in(0,2)\, \forall n\ge
\Phi\,\,  \left(a_n \le \eps\right),
\]
where 
\beq
\Phi:=\Phi(\eps,M,\theta,\gamma)=\theta\left(\gamma\left(\frac\eps 2\right)+1+\left\lceil\ln\left(\frac{2M}\eps\right)\right\rceil\right)+1.  \label{def-Phi-quant-liu-1}
\eeq
and $M\in\Z_+$ is an upper bound on $(a_n)$.
\item\label{quant-liu-2}  If $\lambda_n\in (0,1)$ for all $n\ge 2$ and $\ds \prod_{n=1}^\infty (1-\lambda_{n+1})=0$ with rate of convergence $\theta$, then
\[
\forall \eps\in(0,2)\, \forall n\ge
\Phi\,\,  \left(a_n \le \eps\right), \label{conclusion-quant-liu-2}
\]
where 
\beq
\Phi:=\Phi(\eps,M,\theta,\gamma,D)=\theta\left(\frac{D\eps}{2M}\right)+1,
\eeq
$M\in\Z_+$ is an upper bound on $(a_n)$, and 
\beq
0<D\le \prod_{n=1}^{\gamma(\eps/2)} (1-\lambda_{n+1}).
\eeq
\ee
\end{lemma}
\begin{proof}
\be
\item Follow the proof of  \cite[Lemma 9]{Leu07}.
\item  The proof of (ii) is basically contained in the proof of  \cite[Lemma 9]{Leu07}. For sake of completeness we give it here. We denote $P_n:=\ds \prod_{k=1}^n(1-\lambda_{k+1})$ for all $n\ge 1$. 
Let $\eps\in(0,2)$ and define 
\beq
N:=\gamma\left(\frac\eps 2\right)+1. \label{def-N}
\eeq
Applying \cite[Lemma 8]{Leu07} with $n:=N$, it follows that for all $m\ge 1$,
\bua
a_{N+m}&\leq & \left[\prod_{j=N}^{N+m-1}(1-\lambda_{j+1})\right]a_N+\sum_{j=N}^{N+m-1}b_j=
\frac{P_{N+m-1}}{P_{N-1}}\cdot a_N+\left(s_{\gamma\left(\frac\eps 2\right)+m}-s_{\gamma\left(\frac\eps 2\right)}\right)\\
&\leq & \frac{MP_{N+m-1}}{P_{N-1}}+\frac\eps 2.
\eua
Let 
\beq
L:=\Phi-N=\theta\left(\frac{D\eps}{2M}\right)+1-N.
\eeq
Then for all $m\ge L$, we have that $\ds N+m-1\ge \theta\left(\frac{D\eps}{2M}\right)$, hence
\bua
\frac{MP_{N+m-1}}{P_{N-1}} &\le & \frac{D\eps}{2P_{N-1}}\le \frac{\eps}2.
\eua

This also implies that $L\ge 1$ since, otherwise, 
\[ 1\le M\le \frac{MP_{N+L-1}}{P_{N-1}}\le \frac{\eps}{2}\] 
contradicting $\varepsilon\in (0,2).$ Hence the lemma follows.
\ee
\end{proof}

\section{Effective rates of asymptotic regularity}\label{effective-as-reg} 

The first step towards proving the convergence of the Halpern iterations is to obtain the so-called `asymptotic regularity' and this can be done in the very general setting of $W$-hyperbolic spaces.

Asymptotic regularity is a very important concept in metric fixed-point theory, formally introduced by Browder and Petryshyn in \cite{BroPet66}. A mapping $T$ of a metric space $(X,d)$ into itself is said to be asymptotically regular if $\displaystyle\limn d(x_n,Tx_n)=0$ for all $x\in X$, where $x_n:=T^nx$ is the Picard iteration starting with $x$. We shall say that a  sequence $(y_n)$  in $X$ is \defnterm{asymptotically regular} if $\limn d(y_n,Ty_n)=0$.  A rate of convergence of $(d(y_n,Ty_n))_n$ 
towards $0$ will be called a rate of asymptotic regularity.

The following two propositions provide effective rates of asymptotic regularity for the Halpern iteration. 
Proposition  \ref{Halpern-rate-as-reg-hyperbolic} generalizes to W-hyperbolic spaces a result obtained by the second author for Banach spaces \cite{Leu07}. Proposition \ref{Halpern-rate-as-reg-hyperbolic-1} is new even for the case of Banach spaces.

Let $(X,d,W)$ be a W-hyperbolic space,  $C\se X$ be a  bounded convex subset with diameter $d_C$, $T:C\to C$ be nonexpansive and $(x_n)$ given by (\ref{def-Halpern-iteration}).

\bprop\label{Halpern-rate-as-reg-hyperbolic}
 Assume that $(\lambda_n)$ satisfies (C1), (C2) and (C3). Then $(x_n)$ is asymptotically regular and $\limn d(x_n,x_{n+1})=0$. \\
Furthermore, if $\alpha$ is a rate of convergence of $(\lambda_n)$, $\beta$ is  a Cauchy modulus of $s_n:=\ds\sum_{i=1}^n|\lambda_{i+1}-\lambda_i|$ and $\theta$ is a rate of divergence
of $\ds\sum_{n=1}^\infty \lambda_{n+1}$, then for all $\eps\in (0,2)$, 
\bua
\forall n\ge \tilde{\Phi}\,\, \left(d(x_n,x_{n+1})\leq \eps \right) \quad\text{and}\quad
\forall n\ge \Phi \,\, \left(d(x_n,Tx_n)\leq \eps \right),
\eua
where 
\bea
\tilde{\Phi}:=\tilde{\Phi}(\eps,M,\theta,\beta)&:=& \theta\left(\beta\left(\frac\eps {4M}\right)+1+\left\lceil\ln\left(\frac{2M}\eps\right)\right\rceil\right)+1,  \label{def-tilde-Phi-Halpern}\\ 
\Phi:=\Phi(\eps,M,\theta,\alpha,\beta) & = & \max\left\{\theta\left(\beta\left(\frac\eps {8M}\right)+1+\left\lceil\ln\left(\frac{4M}\eps\right)\right\rceil\right)+1,\alpha\left(\frac\eps {4M}\right)\right\}, \label{def-Phi-Halpern}
\eea
 with $M\in\Z_+$ such that $M\ge d_C$.
\eprop
\begin{proof}
See Section \ref{proof-asymptotic-regularity}.
\end{proof}

Thus, we obtain an effective rate of asymptotic regularity $\Phi(\eps,M,\theta,\alpha,\beta)$  which depends only on the error $\varepsilon$, on an upper bound $M$ on the diameter $d_C$ of $C$, and on $(\lambda_n)$ via $\alpha,\beta,\theta$.  In particular, the  rate $\Phi$ does not depend on $u,x$ or $T$, so 
Proposition 
\ref{Halpern-rate-as-reg-hyperbolic} provides a quantitative version of the 
main theorem in \cite{AoyEshTak07}. Note that what is called `property I'
and `property S' in \cite{AoyEshTak07} has been studied under the name of 
`axioms (W2) and (W4)' in \cite{Koh05}.

\bprop\label{Halpern-rate-as-reg-hyperbolic-1}
Assume that $\lambda_n\in (0,1)$ for all $n\ge 2$ and that $(\lambda_n)$ satisfies (C1), (C2) and (C4). Then $(x_n)$ is asymptotically regular and $\limn d(x_n,x_{n+1})=0$. \\
Furthermore, if $\alpha$ is a rate of convergence of $(\lambda_n)$, $\beta$ is  a Cauchy modulus of $s_n:=\ds\sum_{i=1}^n|\lambda_{i+1}-\lambda_i|$ and $\theta$ is a rate of convergence of $\ds \prod_{n=1}^\infty (1-\lambda_{n+1})=0$ towards $0$, then for all $\eps\in (0,2)$, 
\bua
\forall n\ge \tilde{\Phi}\,\, \left(d(x_n,x_{n+1})\leq \eps \right) \quad\text{and}\quad
\forall n\ge \Phi \,\, \left(d(x_n,Tx_n)\leq \eps \right),
\eua
where 
\bea
\tilde{\Phi}(\eps,M,\theta,\beta,D)&:=& \theta\left(\frac{D\eps}{2M}\right)+1,
\label{def-tilde-Phi-Halpern-1}\\ 
\Phi(\eps,M,\theta,\alpha,\beta,D) & = & \max\left\{\theta\left(\frac{D\eps}{4M}\right)+1,\alpha\left(\frac\eps {4M}\right)\right\}, \label{def-Phi-Halpern-1}
\eea
with $M\in\Z_+$ such that $M\ge d_C$ and $0<D\le \ds \prod_{n=1}^{\beta(\eps/4M)}(1-\lambda_{n+1})$.
\eprop
\begin{proof}
Follow the proof of Proposition \ref{Halpern-rate-as-reg-hyperbolic}, applying Lemma \ref{quant-liu}.(\ref{quant-liu-2}) instead of Lemma \ref{quant-liu}.(\ref{quant-liu-1}).
\end{proof}

That we even get full rates of convergence in Propositions \ref{Halpern-rate-as-reg-hyperbolic}, \ref{Halpern-rate-as-reg-hyperbolic-1} is due to the fact that the original proof of asymptotic regularity is essentially constructive. For such proofs, the requirement of the statement to be proved to have the form $\forall x\exists y\,A_{qf}(x,y)$ with quantifier-free $A_{qf}$, which is crucial for ineffective proofs, is not needed (note that the Cauchy
property is a $\forall\exists\forall$-statement). This is because we do not have to preprocess the proof using some negative translation (which maps proofs with classical logic into ones with constructive logic only) and can directly apply proof-theoretic techniques such as (an appropriate monotone form of) Kreisel's so-called modified realizability interpretation. Logical metatheorems covering such situations are proved in \cite{GerKoh06}. As a consequence of getting full rates of convergence in Propositions \ref{Halpern-rate-as-reg-hyperbolic}, \ref{Halpern-rate-as-reg-hyperbolic-1} one then also has to strengthen the premises on the convergence of $(\lambda_n)$ and $\sum\limits^{\infty}_{n=1} |\lambda_{n+1}-\lambda_n|$ by full rates of convergence $\alpha,\beta.$ If we would interpret the proof as an ineffective one using the metatheorems from \cite{Koh05}, then one would only get a rate of metastability in the conclusion but also would
only need rates of metastability for these premises (note that $\sum\limits^{\infty}_{n=1} \lambda_n =\infty$ is a $\forall\exists$-statement so that there is no difference here between a
full rate and a rate of metastability).

As an immediate consequence of Proposition \ref{Halpern-rate-as-reg-hyperbolic-1}, for $\ds \lambda_n=\frac 1{n+1}$ we get a quadratic (in $1/\eps$) rate of asymptotic regularity.  For Banach spaces, this rate of asymptotic regularity was obtained by the first author in \cite{Koh11}. In \cite{Leu07}, the second author obtained an exponential rate of asymptotic regularity due to the fact that he used the version for Banach spaces of Proposition \ref{Halpern-rate-as-reg-hyperbolic}, which needs a rate of divergence of $\ds \sum_{n=1}^\infty \frac 1{n+1}$.

\bcor\label{Halpern-rate-as-reg-hyperbolic-1n}
Assume that $\ds\lambda_n=\frac{1}{n+1}$ for all $n\geq 1$. Then 
for all $\eps\in (0,1)$, 
\bea
\forall n\ge \tilde{\Psi}(\eps,M)\,\, \left(d(x_n,x_{n+1})\leq \eps \right) \quad\text{and}\quad
\forall n\ge \Psi(\eps,M) \,\, \left(d(x_n,Tx_n)\leq \eps \right),
\eea
where 
\bea
\tilde{\Psi}(\eps,M):=  \left\lceil \frac{2M}{\eps}+\frac{8M^2}{\eps^2}\right\rceil - 1\quad\text{and}\quad
\Psi(\eps,M) :=  \left\lceil \frac{4M}{\eps}+\frac{16M^2}{\eps^2}\right\rceil-1,
\eea
with $M\in\Z_+$ such that $M\ge d_C$.
\ecor
\begin{proof}
Obviously, $\ds\lim_{n\to\infty}\frac{1}{n+1}=0$ with a rate of convergence $\ds \alpha(\eps)=\left\lceil\frac{1}\eps\right\rceil-1\ge 1$. As we have already seen, $\ds \theta(\eps):=\left\lceil\frac{2}\eps\right\rceil-2$ is a rate of convergence of $\ds \prod_{n=1}^\infty\left(1-\frac{1}{n+2}\right)$ towards $0$.
Furthermore, 
\bua
s_n:=\sum_{k=1}^n\left|\frac 1{k+2}-\frac 1{k+1}\right|=\frac12-\frac 1{n+2}.
\eua
It follows easily that $\limn s_n=1/2$ with Cauchy modulus $\ds \beta(\eps):=\begin{cases}\ds \left\lceil 1/\eps\right\rceil-1& \ds \text{if } \eps\ge 1/2\\
\ds \left\lceil 1/\eps\right\rceil-2 & \ds \text{if } \eps < 1/2\end{cases}$. 
 
Finally, $\ds \prod_{n=1}^{\beta(\eps/4M)}\left(1-\frac1{n+2}\right)=\frac{2}{\lceil 4M/\eps\rceil}
$, as $\ds \frac{\eps}{4M}<\frac 12$, so we can take $\ds D:=\frac{2}{\lceil 4M/\eps\rceil}$. Apply now Proposition \ref{Halpern-rate-as-reg-hyperbolic-1} and use the fact that $\lceil x\rceil \le x+1$ to get the result.
\end{proof}

\section{Proof of Proposition \ref{Halpern-rate-as-reg-hyperbolic}}\label{proof-asymptotic-regularity}

The following lemma collects some useful properties of Halpern iterations that hold for unbounded $C$ too.

\begin{lemma}\label{lemma-Halpern-hyp}
Assume that $(x_n)$ is the Halpern iteration starting with $x\in C$. Then
\be
\item For all $n\geq 0$,
\beq
d(x_{n+1},Tx_n) = \lambda_{n+1}d(Tx_n,u)\quad\text{and}\quad d(x_{n+1},u)=(1-\lambda_{n+1})d(Tx_n,u).
\eeq
\item For all $n\geq 0$,
\bea
d(Tx_n,u)&\leq & d(u,Tu)+d(x_n,u), \label{H-Txn-u}\\
d(x_n,Tx_n) &\leq & d(x_{n+1},x_n)+ \lambda_{n+1}d(Tx_n,u), \label{H-Txn-xn}\\
d(x_{n+1},u) &\leq & (1-\lambda_{n+1})\big(d(u,Tu)+d(x_n,u)\big), \label{H-xn+1-u}\\
d(x_{n+1},x_n) &\leq & \lambda_{n+1}d(x_n,u)+(1-\lambda_{n+1})d(Tx_n,x_n). \label{H-xn+1-xn-ineq2}
\eea
\item For all $n\geq 1$,
\bea
d(x_{n+1},x_n) &\leq &  (1-\lambda_{n+1})d(x_n,x_{n-1})+|\lambda_{n+1}-\lambda_n|\,d(u,Tx_{n-1}) \label{H-xn+1-xn-ineq1}.
\eea
\item  If $(x_n)$ is bounded, then  $(Tx_n)$ is also bounded. Moreover, if $M\geq  d(u,Tu)$ and $M\ge d(x_n,u)$ for all $n\ge 0$,
\bea
d(x_n,Tx_n) &\leq &  d(x_{n+1},x_n) +2M\lambda_{n+1} \text{ and}\label{H-ineq-Txn-xn}\\
d(x_{n+1},x_n) & \leq &  (1-\lambda_{n+1})d(x_n,x_{n-1})+2M|\lambda_{n+1}-\lambda_n| \label{H-ineq-xn+1-xn}
\eea
for all $n\geq 1$.
\ee 
\end{lemma}
\begin{proof}
\be
\item By (\ref{prop-xylambda}).
\item
\bua
d(Tx_n,u)&\leq & d(u,Tu)+d(Tu,Tx_n) \leq d(u,Tu)+d(x_n,u),\\
d(x_n,Tx_n) &\leq & d(x_{n+1},x_n)+d(Tx_n,x_{n+1})=d(x_{n+1},x_n)+ \lambda_{n+1}d(Tx_n,u)\\
d(x_{n+1},u) &= & (1-\lambda_{n+1})d(Tx_n,u) \le (1-\lambda_{n+1})\big(d(u,Tu)+d(x_n,u)\big)\\
d(x_{n+1},x_n) &\le & \lambda_{n+1}d(x_n,u)+(1-\lambda_{n+1})d(x_n,Tx_n) \quad \text{~by (W1)}.
\eua
\item 
\bua
d(x_{n+1},x_n) &=& d(\lambda_{n+1}u\oplus (1-\lambda_{n+1})Tx_n,\lambda_nu\oplus(1-\lambda_n)Tx_{n-1})\\
&\leq & d(\lambda_{n\!+\!1}u\oplus (1\!-\!\lambda_{n\!+\!1})Tx_n,\lambda_{n\!+\!1}u\oplus(1\!-\!\lambda_{n\!+\!1})Tx_{n\!-\!1})\!\\
&& + d(\lambda_{n+1}u\oplus(1-\lambda_{n+1})Tx_{n-1},\lambda_nu\oplus(1-\lambda_n)Tx_{n-1})\\
&\leq &  (1-\lambda_{n+1})d(Tx_n,Tx_{n-1})+ |\lambda_{n+1}-\lambda_n|d(u,Tx_{n-1})\\
&&\text{by (W4) and (W2)}\\
&\leq &  (1-\lambda_{n+1})d(x_n,x_{n-1})+|\lambda_{n+1}-\lambda_n|d(u,Tx_{n-1}).
\eua
\item is an easy consequence of (ii), (iii).
\ee
\end{proof}

In the following, we give the proof of Proposition \ref{Halpern-rate-as-reg-hyperbolic}.

Let us consider the sequences
\[a_n:=d(x_n,x_{n-1}),  \quad b_n:=2M|\lambda_{n+1}-\lambda_n|\]
By (\ref{H-ineq-xn+1-xn}), we get that 
\[a_{n+1}\leq (1-\lambda_{n+1})a_n+b_n \quad \text{for all~~}n\geq 1.\]
Moreover, $\ds\sum_{n=1}^\infty \lambda_{n+1}$ is divergent with rate of divergence $\theta$ and it is easy to see that 
\[\gamma:(0,\infty)\to \Z_+, \quad \gamma(\eps):=\beta\left(\frac\eps{2M}\right)\]
is a Cauchy modulus of $s_n:=\ds\sum_{i=1}^nb_i$.

Thus, the hypotheses of Lemma \ref{quant-liu}.(\ref{quant-liu-1}) are satisfied, so we can apply it to get that for all $\eps\in(0,2)$ and for all $n\geq \tilde{\Phi}(\eps,M,\theta,\beta)$
\beq
d(x_n,x_{n-1})\le \eps,\label{Halpern-final-1}
\eeq
where 
\bua
\tilde{\Phi}(\eps,M,\theta,\beta)&:=& \theta\left(\beta\left(\frac\eps {4M}\right)+1+\left\lceil\ln\left(\frac{2M}\eps\right)\right\rceil\right)+1.
\eua
By (\ref{H-ineq-Txn-xn}), for all $n\geq 2$,
\beq
d(x_{n-1},Tx_{n-1}) \leq  d(x_n,x_{n-1})+2M\lambda_n. \label{Halpern-final-0}
\eeq
Since  $\alpha$ is a rate of convergence of $(\lambda_n)$ towards $0$, we get that 
\beq
2M\lambda_{n}\le \frac\eps 2 \quad\text{for all } n\geq \alpha\left(\frac\eps {4M}\right).\label{Halpern-final-2}
\eeq
Combining (\ref{Halpern-final-1}), (\ref{Halpern-final-0}) and (\ref{Halpern-final-2}) it follows that 
\[d(x_{n-1},Tx_{n-1})\le \eps\]
for all $n\geq \ds \max\left\{\tilde{\Phi}\left(\frac{\eps}2,M,\theta,\beta\right),\alpha\left(\frac\eps {4M}\right)\right\}$, so the conclusion of the theorem follows.

\section{Elimination of Banach limits}\label{elim-Ban-lim}

Let us recall that a \defnterm{Banach limit} \cite{Ban32} is a linear functional $\mu: \ell^\infty\to \R$ satisfying the following properties: 
\be
\item $\mu((x_n))\ge 0$ if  $x_n\ge 0$ for all $n\ge 0$;
\item $\mu(\mathbf{1})=1$;
\item $\mu((x_n))=\mu((x_{n+1}))$. 
\ee
Here $\mathbf{1}$ is the sequence $(1,1,\ldots )$ and $(x_{n+1})$ is the sequence $(x_1,x_2,\ldots)$. 

As we have already said, to prove the existence of Banach limits one needs the axiom of choice (see, e.g., \cite{Suc67}). Banach limits are mainly used in Saejung's convergence proof to get the following: 

\blem\label{lemma-Banach-limits-Saejung}\cite{ShiTak97}
Let $(a_k)\in\ell^\infty$ and $a\in\R$ be such that $\mu((a_k))\le a$ for all Banach limits 	$\mu$ and  $\lsupk(a_{k+1}-a_k)\le 0$. Then $\lsupk a_k\le a$.
\elem

Given a sequence $(a_k)_{k\ge 1}$, consider for all $n,p\ge 1$ the following average
\beq
C_{n,p}((a_k))=\frac1p\sum_{i=n}^{n+p-1}a_i.
\eeq
For simplicity we shall write $C_{n,p}(a_k)$.

Lemma \ref{lemma-Banach-limits-Saejung} is proved using a result that goes back to Lorentz  \cite{Lor48}.

\blem\label{char-Banach-lim-Cnp-leq-a}
Let $(a_k)\in\ell^\infty$ and $a\in\R$. The following are equivalent:
\be
\item $\mu((a_k))\le a$ for all Banach limits $\mu$.
\item For all $\eps>0$ there exists $P\ge 1$ such that $C_{n,p}(a_k)\le a+\eps$ for all $p\ge P$ and  $n\ge 1$.
\ee
\elem

In fact, one only needs the implication `(i) $\Ra$ (ii)' which is
established in \cite{ShiTak97} using the following sublinear
functional 
\[ q:l^{\infty}\to\R, \quad q((a_k)):=\limsup_{p\to\infty}\ \sup_{n\ge 1} \frac{1}{p}
\sum^{n+p-1}_{i=n} a_i=\limsup_{p\to\infty}\ \sup_{n\ge 1}C_{n,p}(a_k).\]

\noindent Now fix $(a_k)\in l^{\infty}$ and use the Hahn-Banach theorem to show the
existence of a linear functional $\mu:l^{\infty}\to\R$ such that
$\mu \le q$ and $\mu((a_k))=q((a_k)).$ Then $\mu$ is a Banach limit and
so -- by (i) -- $q((a_k))=\mu((a_k))\le a$ which gives (ii). Our
elimination of the use of the Banach limit $\mu$ was obtained in two
steps: first, the proof that -- for the sequence in question in the proof
from \cite{Sae10} -- the fact $\mu ((a_k))\le a$ holds for all Banach
limits $\mu$ could be modified to directly showing this for $q$ instead of
$\mu.$ This already established the actual elimination of the  use of
the axiom of choice hidden in the application of the Hahn-Banach theorem
(for the nonseparable space $l^{\infty}$) since the existence of $q$
follows by just using uniform arithmetical comprehension in the form of an
operator $E:\N^{\N}\to \{ 0,1\}$ defined by
\[ E(f)=0 \leftrightarrow \forall n\in\N (f(n)=0),\]
that is needed (and sufficient) to form both the `$\sup$' as well as the
`$\limsup$' in the definition of $q$ (as a function in $(a_k)$).
Using an argument due to Feferman \cite{Fef77}, the use of $E$ can
(over the system used to formalize the overall proof) be eliminated in
favor of ordinary (non-uniform) arithmetic comprehension
\[ \forall f:\N^2\to\N \,\exists g:\N\to\N\,\forall k\in\N\
\big( g(k)=0\leftrightarrow \forall n\in\N\,(f(k,n)=0)\big),\]
which is covered (as a very special case of general comprehension over
numbers) by the existing logical metatheorems and results in extractable
bounds of restricted complexity, namely bounds that are definable by
primitive recursive functionals in the extended sense of G\"odel's
calculus $T$ \cite{God58} (which, however, contains the famous
so-called Ackermann function), though in general not of ordinarily
primitive recursive type. 

In order to get a bound having the latter much more restricted complexity we -- in a second step -- also eliminated the use of $q$ in favor of just elementary lemmas on the finitary objects $C_{n,p}$. In the following, rather than going through these two steps separately, we just present the resulting elementary lemmas on the averages $C_{n,p}$ which we will need later. The first lemma collects some obvious facts.

\blem\label{prop-cNP-Bl}
Let  $(a_k),(b_k)$ be sequences of real numbers and $\alpha\in\R$.
\be
\item\label{Cnp-incr} If $a_k\le b_k$  for all $k\ge N$, then  $C_{n,p}(a_k)\le C_{n,p}(b_k)$ for all $n\ge N$ and  $p\ge 1$.
\item\label{Cnp-an-const} If $a_k=c\in \R$ for all $k\ge N$, then $C_{n,p}(a_k)=c$ for all  $n\ge N$ and  $p\ge 1$.
\item\label{Cnp-linear} For all $n,p\ge 1$, $C_{n,p}(a_k+b_k)=C_{n,p}(a_k)+C_{n,p}(b_k)$ and $C_{n,p}(\alpha a_k)=\alpha C_{n,p}(a_k)$.
\ee
\elem

\blem\label{Cnp-basic-prop-1-rate}
Let $(a_k)$ be a sequence of real numbers, $a\in\R$ and $P:(0,\infty)\to\Z_+$ be such that 
\beq
\forall\eps>0\,\forall n\ge 1\,\,\big(C_{n,P(\eps)}(a_k) \le a+\eps\big).
\eeq
Assume that $\lsupk(a_{k+1}-a_k)\le 0$ with effective rate $\theta$.

Then $\lsupk a_k\le a$ with effective rate $\psi$, given by
\beq
\psi(\eps,P,\theta)= \theta\left(\frac{\eps}{\tilde{P}+1}\right)+\tilde{P},
\eeq
where $\ds \tilde{P}:=P\left(\frac{\eps}2\right)$.
\elem
\begin{proof}
By hypothesis,
\bua
C_{n,\tilde{P}}(a_k) \le  a+\frac{\eps}2\quad  \text{for all }n\ge 1,
\eua
and
\bua
a_{k+1}-a_k\le  \frac{\eps}{\tilde{P}+1} \quad \text{for all }k\ge \theta\left(\frac{\eps}{\tilde{P}+1}\right).
\eua
Let $n\ge \psi(\eps,P,\theta)$. Then  $n=n_0+\tilde{P}$ for some $\ds n_0\ge \theta\left(\frac{\eps}{\tilde{P}+1}\right)$. We get that for each $i=0,\ldots, \tilde{P}-1$, 
\bua
a_{n} &=& a_{n_0+\tilde{P}} = a_{n_0+i}+(a_{n_0+i+1}-a_{n_0+i})+(a_{n_0+i+2}-a_{n_0+i+1})+\ldots + (a_{n_0+\tilde{P}}-a_{n_0+\tilde{P}-1})\\
&\le & a_{n_0+i}+\frac{(\tilde{P}-i)\eps}{\tilde{P}+1}.
\eua
By adding the inequalities, we get that 
\bua
\tilde{P}a_{n} &\le & \big(a_{n_0}+a_{n_0+1}+\ldots+a_{n_0+\tilde{P}-1}\big)+\frac{(1+2+\ldots +\tilde{P})\eps}{\tilde{P}+1}\\
&=& \big(a_{n_0}+a_{n_0+1}+\ldots+a_{n_0+\tilde{P}-1}\big)+\frac{\tilde{P}\eps}2,
\eua
hence
\bua
a_{n} &\le & C_{n_0,\tilde{P}}(a_k)+\frac{\eps}2 \leq a+\frac{\eps}2+\frac{\eps}2=a+\eps.
\eua
\end{proof}

\blem\label{Cnp-ak-conv-0} 
Assume that $(a_k)$ is nonnegative and $\limk a_k=0$. Then $\ds\limp C_{n,p}(a_k)=0$ uniformly in $n$. 

Furthermore, if $\vp$ is a rate of convergence of $(a_k)$, then for all $\eps\in (0,2)$,
\[
\forall p\ge P(\eps,\vp,L)\,\forall n\ge 1 \,\,\left(C_{n,p}(a_k)\le \eps\right),
\]
where 
\beq
P(\eps,\vp,L)=\left\lceil\frac{2L\vp(\eps/2)}{\eps}\right\rceil,\eeq
with $L\in\R$ being an upper bound on $(a_k)$.
\elem
\begin{proof}
 Let $\vp, L, \eps$ be as in the hypothesis. We shall denote $P(\eps,\vp,L)$  simply by $P$. Since $\vp$ is a rate of convergence of $(a_k)$, we have that 
$a_k \le \ds\frac{\eps}2$ for all $k\ge \vp(\eps/2)$. Furthermore, 
\beq
\frac{L\vp(\eps/2)}p\le \ds\frac{\eps}2 \quad\text{ for all } p\ge P.
\eeq
Let $p\ge P$ and $n\ge 1$. We have two cases:
\be
\item $n\ge \vp(\eps/2)$. Then 
\bua
C_{n,p}(a_k) &=& \frac1p\sum_{i=n}^{n+p-1}a_i \le  \frac1p\cdot \ds\frac{p\eps}2=\ds\frac{\eps}2<\eps.
\eua
\item $n<\vp(\eps/2)$. Then 
\bua
C_{n,p}(a_k) &\le & \frac1p\sum_{i=n}^{\vp(\eps/2)-1}a_i+\frac1p\sum_{i=\vp(\eps/2)}^{\vp(\eps/2)+p-1}a_i 
\le \frac{\left(\vp(\eps/2)-n\right)L}p+ \frac{\eps}2 \le \frac{\vp(\eps/2)L}p+ \frac{\eps}2\\
&\le & \eps.
\eua
\ee
Thus, we have proved that $C_{n,p}(a_k) \le \eps$ for all $p\ge P$ and  $n\ge 1$.
\end{proof}

\section{Quantitative properties of an approximate fixed point sequence}
\label{section-Banach}

In the following, $X$ is a complete CAT(0) space, $C\se X$ is a bounded convex closed subset and $T:C\to C$ is a nonexpansive mapping. We assume that $C$ is bounded with diameter $d_C$ and consider $M\in\Z_+$ with $M\ge d_C$.

For $t\in (0,1)$ and $u\in C$, define 
\beq
T_t^u:C\to C, \quad T_t^u(y)=tu\oplus(1-t)Ty. \label{averaged-u}
\eeq

It is easy to see that $T_t^u$ is a contraction with contractive constant $L:=1-t$, so it has a unique fixed point $z_t^u\in C$ by Banach's Contraction Mapping Principle. Hence, $z_t^u$ is the unique solution of the fixed point equation
\beq
z_t^u=tu\oplus(1-t)Tz_t^u.\label{Halpern-continuous}
\eeq

\bprop\label{properties-yn-ztu-gammant}
Let $(y_n)$ be a sequence in $C$, $u\in C$, $t\in (0,1)$, and $(z_t^u)$ be 
defined by (\ref{Halpern-continuous}). Define for all $n\ge 1$
\beq
\gamma_n^t:=(1-t)d^2(u,Tz_t^u)-d^2(y_n,u).\label{def-gamma-n-t}
\eeq
\be
\item For all $n\ge 1$,
\beq
d^2(y_n, z_t^u) \le  d^2(y_n,u)+\frac1{t}a_n-(1-t)d^2(u,Tz_t^u),\label{ineq-yn-zt}
\eeq
where 
\beq
a_n:=d^2(y_n,Ty_n)+2Md(y_n,Ty_n).\label{def-an-yn}
\eeq
\item\label{Cnp-gammant-1} If $(y_n)$ is asymptotically regular with rate of asymptotic regularity $\vp$, then for all $\eps\in(0,2)$, 
\beq
\forall p\ge P(\eps,t,M,\vp)\,\forall m\ge 1\,\, \left(C_{m,p}(\gamma_n^t) \le \eps \right),
\eeq 
where
\beq
P(\eps,t,M,\vp)=\left\lceil\frac{6M^2}{t\eps}\vp\left(\ds\frac{t\eps}{6M}\right)\right\rceil. \label{def-Pteps}
\eeq
\item\label{Cnp-gammant-2} Assume that $(y_n)$ is asymptotically regular and  $\ds \limn d(y_n,y_{n+1})=0$. Then $\ds \lsupn \gamma_n^t\le 0$.
Furthermore, if $\vp$ is a rate of asymptotic regularity of $(y_n)$,  and $\tilde{\vp}$ is a rate of convergence of  $(d(y_n,y_{n+1}))$ towards $0$, then  $\ds \lsupn \gamma_n^t\le 0$ with effective rate $\psi$, defined by
\beq
\psi(\eps,t,M,\vp,\tilde{\vp})=\tilde{\vp}\left(\frac{\eps}{2M(P\left(\eps/2,t,M,\vp\right)+1)}\right)+P\left(\eps/2,t,M,\vp\right),
\eeq
with $P$ given by (\ref{def-Pteps}).
\ee
\eprop
\begin{proof}
For simplicity, we shall denote $z_t^u$ by $z_t$.
\be
\item We get that for all $n\ge 1$, 
\bua
d^2(y_n, z_{t}) &=& d^2(y_n, tu\oplus(1-t)Tz_{t})\\
&\leq & td^2(y_n,u)+(1-t)d^2(y_n,Tz_{t})-t(1-t)d^2(u,Tz_{t})\quad \text{by (\ref{CAT0-ineq-t}})\\
&\leq & td^2(y_n,u)+(1-t)\big(d(y_n,Ty_n)+d(Ty_n,Tz_{t})\big)^2-t(1-t)d^2(u,Tz_{t})\\
&& \text{by the triangle inequality}\\
&\leq & td^2(y_n,u)+(1-t)\big(d(y_n,Ty_n)+d(y_n,z_{t})\big)^2-t(1-t)d^2(u,Tz_{t})\\
 && \text{by the nonexpansiveness of }T\\
&= &  td^2(y_n,u)+ (1-t)d^2(y_n,Ty_n)+2(1-t)d(y_n,Ty_n)d(y_n,z_{t})\\
&& + (1-t)d^2(y_n, z_{t})- t(1-t)d^2(u,Tz_{t})\\
&\le & td^2(y_n,u)+ (1-t)d^2(y_n,Ty_n)+2M(1-t)d(y_n,Ty_n)+(1-t)d^2(y_n, z_{t})\\
&& - t(1-t)d^2(u,Tz_{t})
\eua
Thus, for all $n\ge 1$,
\bua
td^2(y_n, z_{t}) &\le & td^2(y_n,u)+ (1-t)d^2(y_n,Ty_n)+2M(1-t)d(y_n,Ty_n)-t(1-t)d^2(u,Tz_{t})\\
&\le & td^2(y_n,u)+ d^2(y_n,Ty_n)+2Md(y_n,Ty_n)-t(1-t)d^2(u,Tz_{t}).
\eua
Hence, (\ref{ineq-yn-zt}) follows.
\item  Let $\eps\in(0,2)$. By (\ref{ineq-yn-zt}), we get that 
\bua
0\le d^2(y_n,u)+\frac1{t}a_n-(1-t)d^2(u,Tz_{t}), 
\eua
 hence $\ds \gamma_n^t\le \frac1{t}\,a_n$ for all $n\ge 1$. It follows by 
Lemma \ref{prop-cNP-Bl}.(\ref{Cnp-incr}) that
\[C_{m,p}(\gamma_n^t)\le C_{m,p}\left(\frac1{t}\,a_n\right)\quad\text{for all }m\ge 1,\, p\ge 1.\]
Furthermore, $\ds \limn a_n=\limn \left(d^2(y_n,Ty_n)+2Md(y_n,Ty_n)\right)=0$ and, given a rate of asymptotic regularity for $(y_n)$, we can easily verify that $\ds \vp\left(\frac{\eps}{3M}\right)$ is a rate of convergence of $(a_n)$ towards $0$. 

Then $\ds \vp\left(\frac{t\eps}{3M}\right)$ is a rate of convergence of $\ds \frac1{t}\,a_n$ towards $0$. Since $\ds L:=\frac{3M^2}{t}$ is an upper bound for $\left(\frac1{t}a_n\right)$, we can apply Lemma \ref{Cnp-ak-conv-0} for this sequence to conclude that
\beq
C_{m,p}\left(\frac1{t}a_n\right) \le \eps \quad\text{for all } p\ge P(\eps,t,M,\vp) \text{ and  }m\ge 1.
\eeq
\item We have that 
\bua
|\gamma_{n+1}^t-\gamma_n^t|&=& |((1-t)d^2(u,Tz_t)-d^2(y_{n+1},u))-((1-t)d^2(u,Tz_t)-d^2(y_n,u))|\\
&=& |d^2(y_n,u)-d^2(y_{n+1},u)|= |d(y_n,u)+d(y_{n+1},u)|\cdot |d(y_n,u)-d(y_{n+1},u)|\\
&\le & 2Md(y_n,y_{n+1}).
\eua
Since $\limn 2Md(y_n,y_{n+1})=0$, we get that
\bua
\lsupn (\gamma_{n+1}^t-\gamma_n^t) \le 0.
\eua
with effective rate $\ds\tilde{\vp}\left(\frac{\eps}{2M}\right)$. Apply (\ref{Cnp-gammant-1}) and Lemma \ref{Cnp-basic-prop-1-rate} to conclude that
\beq
\lsupn \gamma_n^t\le 0.
\eeq
with effective rate $\psi(\eps,t,M,\vp,\tilde{\vp})$.
\ee
\end{proof}

\blem\label{Halpern-d-zt}
Let $u,x\in C$ and $(x_n)$ be the Halpern iteration defined by (\ref{def-Halpern-iteration}). Then for all  $t\in (0,1)$ and  $n\ge 0$,
\bea
\!\!\!\! d^2(x_{n+1},z_t^u) &\le & (1-\lambda_{n+1})d^2(x_n,z_t^u)+\lambda_{n+1}\bigg((1-t)d^2(u,Tz_t^u)-d^2(x_{n+1},u)\bigg)+M^2t.
\label{d2xn+1-zt-v3}
\eea
\elem
\begin{proof}
\bua
d^2(x_{n+1},z_t^u) &\leq & \lambda_{n+1}d^2(u,z_t^u)+(1-\lambda_{n+1})d^2(Tx_n,z_t^u)-\lambda_{n+1}(1-\lambda_{n+1})d^2(u,Tx_n)\\
&& \text{by (\ref{CAT0-ineq-t})  applied to } d^2(x_{n+1},z_t^u)=d^2(\lambda_{n+1}u\oplus(1-\lambda_{n+1})Tx_n,z_t^u)\\
&\le & \lambda_{n+1}d^2(u,z_t^u)-\lambda_{n+1}(1-\lambda_{n+1})d^2(u,Tx_n)+\\
&& +(1-\lambda_{n+1})\bigg(td^2(Tx_n,u)+(1-t)d^2(Tx_n,Tz_t^u)-t(1-t)d^2(u,Tz_t^u)\bigg)\\
&& \text{again by (\ref{CAT0-ineq-t}) applied to } d^2(Tx_n,z_t^u)=d^2(Tx_n,tu\oplus(1-t)Tz_t^u)\\
&\le & \lambda_{n+1}d^2(u,z_t^u)-\lambda_{n+1}(1-\lambda_{n+1})d^2(u,Tx_n)+\\
&& +(1-\lambda_{n+1})\bigg(td^2(Tx_n,u)+(1-t)d^2(x_n,z_t^u)-t(1-t)d^2(u,Tz_t^u)\bigg)\\
&& \text{by the nonexpansiveness of } T \\
&=& (1-\lambda_{n+1})(1-t)d^2(x_n,z_t^u)+\\
&&+d^2(Tx_n,u)\bigg((1-\lambda_{n+1})t-\lambda_{n+1}(1-\lambda_{n+1})\bigg)+\\
&& + \lambda_{n+1}(1-t)^2d^2(u,Tz_t^u)- (1-\lambda_{n+1})t(1-t)d^2(u,Tz_t^u)\\
&& \text{since }d(u,z_t^u)=(1-t)d(u,Tz_t^u)\\
&=& (1-\lambda_{n+1})(1-t)d^2(x_n,z_t^u)+\\
&&+\lambda_{n+1}\bigg((1-t)d^2(u,Tz_t^u)-(1-\lambda_{n+1})^2d^2(Tx_n,u)\bigg)\\
&&+d^2(Tx_n,u)\bigg((1-\lambda_{n+1})t-\lambda_{n+1}(1-\lambda_{n+1})+\lambda_{n+1}(1-\lambda_{n+1})^2
\bigg)+\\
&& +d^2(u,Tz_t^u)\bigg(\lambda_{n+1}(1-t)^2-(1-\lambda_{n+1})t(1-t)-\lambda_{n+1}(1-t)\bigg)
\eua
\bua
&=& (1-\lambda_{n+1})(1-t)d^2(x_n,z_t^u)+\lambda_{n+1}\bigg((1-t)d^2(u,Tz_t^u)-d^2(x_{n+1},u)\bigg)\\
&& +d^2(Tx_n,u)\bigg(t-\lambda_{n+1}t+\lambda_{n+1}^3-\lambda_{n+1}^2\bigg)+d^2(u,Tz_t^u)(t^2-t)\\
&& \text{since }d(x_{n+1},u)=(1-\lambda_{n+1})d(Tx_n,u)\\
&\le & (1-\lambda_{n+1})(1-t)d^2(x_n,z_t^u)+\lambda_{n+1}\bigg((1-t)d^2(u,Tz_t^u)-d^2(x_{n+1},u)\bigg)+\\
&& + td^2(Tx_n,u)\\
&\le & (1-\lambda_{n+1})d^2(x_n,z_t^u)+\lambda_{n+1}\bigg((1-t)d^2(u,Tz_t^u)-d^2(x_{n+1},u)\bigg)+M^2t.
\eua
\end{proof}

In \cite{Bro67}, Browder showed that for Hilbert spaces $X$ and $z^u_t$ defined as above one has, for $t\to 0,$ the strong convergence of $z^u_t$ towards the fixed point of $T$ that is closest to $u.$
Halpern \cite{Hal67} gave a much more elementary proof of this result. In fact, it follows from his proof that the strong convergence of $(z^u_{t_k})_k$ holds for any nonincreasing sequence $(t_k)$ in $(0,1)$ (while the limit in general will not be a fixed point of $T$ unless $t_k$ converges towards $0$). In \cite{Koh11}, the first author extracted explicit and highly uniform rates of metastability from both proofs (again effective rates of
convergence are ruled out on general grounds, see \cite{Koh11}). In \cite{Kir03}, Kirk showed that Halpern's proof goes through (essentially unchanged) in the context of CAT$(0)$ spaces. Consequently, this also holds for the bound extracted from Halpern's proof in \cite{Koh11} (for the Hilbert ball this is already 
due to \cite{GoeRei84}):

\bprop\label{quant-Browder-Halpern}
Let $(t_k)$ be a nonincreasing sequence in $(0,1).$
Then for all $\varepsilon >0$ and  $g:\N\to\N$ the following holds
\[\exists K_0\le K(\varepsilon,g,M)\,\forall i,j\in [K_0,K_0+g(K_0)] \
\big( d(z^u_{t_i},z^u_{t_j})\le \varepsilon\big), \]
where
\beq
K(\varepsilon,g,M):=\tilde{g}^{(\lceil M^2/\varepsilon^2\rceil)}(0),
\label{def-K-tk-1k+1}
\eeq
with $\tilde{g}(k):=k+g(k).$
\eprop
\begin{proof}
For the case of $X$ being a Hilbert space, Proposition
\ref{quant-Browder-Halpern} is
proved in \cite{Koh11}. Things extend unchanged to the CAT(0)-setting
with the same reasoning as in \cite{Kir03}.
\end{proof}
\begin{remark}
\begin{enumerate}
\item
Reasoning as in \cite{Koh11}, Proposition \ref{quant-Browder-Halpern} implies the following rate of
metastability for sequences $(t_k)$ that are not necessarily nonincreasing: let $(t_k)_{k\ge 0}$ be a sequence in $(0,1)$ that converges towards $0$ with rate of convergence $\beta$ and $\chi:\N\to\N$ be defined by
$\chi(k)=\beta\left(\frac{1}{k+1}\right)$, hence
\[ \forall k\in\N\,\forall i\ge \chi(k)\ \left(t_i\leq
\frac{1}{k+1}\right). \]
Finally, let $h:\N\to\N$ be such that $t_k\ge \frac{1}{h(k)+1}$ for all
$k\in\N.$
Then for all $\varepsilon >0$ and $g:\N\to\N$ the following holds
\[\exists K_0\le K( \varepsilon,g,M,\chi)\,\forall i,j\in [K_0,K_0+g(K_0)] \
\big( d(z^u_{t_i},z^u_{t_j})\le \varepsilon\big), \]
where
\[ K(\varepsilon,g,M,\chi,h):=\chi^+\big(
g^{(\lceil 4M^2/\varepsilon^2\rceil)}_{h,\chi}(0)\big), \ \mbox{with} \
g_{h,\chi}(k):=\max\{ h(i)\mid i\le \chi(k)+g(\chi(k))\}. \]
\item
Instead of a rate of convergence $\beta$ it suffices in `(i)' above
to have a rate of metastability $\beta_g$, hence a mapping $\beta_g$ such
that
\[
\forall k\in\N \,\forall i\in [\beta_g(k),\tilde{g}(\beta_g(k))] \
\left( t_i \le \frac{1}{k+1}\right).
\]
\end{enumerate}
\end{remark}

\section{Proof of Theorem \ref{Halpern-rate}}\label{proof-Halpern-rate}

Let $\eps\in(0,2)$ and $g:\N\to\N$ be fixed. Let $\tilde{\Phi},\Phi$ be as in Proposition \ref{Halpern-rate-as-reg-hyperbolic}. To make the proof easier to read, we shall omit parameters $M, \Phi,\tilde{\Phi}, \theta,\alpha,\beta$ for all the functionals which appear in the following.

Take 
\beq
\eps_0:=\frac{\eps^2}{24(M+1)^2}\,. \label{def-eps0}
\eeq
Then $\eps_0<1$ and
\beq
\eps_0^2+2M\eps_0+ M^2\eps_0\le \eps_0(M+1)^2\le \frac{\eps^2}{24}. \label{eps0-prop}
\eeq

We consider in the sequel $\ds t_k:=\frac{1}{k+1}$, with rate of convergence towards $0$ given by $\ds \gamma(\eps):=\left\lceil\frac{1}{\eps}\right\rceil. $

Denote $z_{t_k}^u$ simply by $z_k^u$ and let
\bua
\gamma_n^k:=\frac{k}{k+1}d^2(u,Tz_k^u)-d^2(x_{n+1},u). 
\eua
Thus, $\gamma_n^k$ is defined as in (\ref{def-gamma-n-t}) by taking $\ds t:=\frac{1}{k+1}$ and $y_n:=x_{n+1}$.

We can apply Propositions \ref{properties-yn-ztu-gammant}.(\ref{Cnp-gammant-2}) and \ref{Halpern-rate-as-reg-hyperbolic} to conclude that $\ds \lsupn \gamma_n^k\le 0$ for each $k\ge 0$,  with effective rate $\chi_k$, given by  
\bua
\chi_k(\eps)&=& \tilde{\Phi}\left(\frac{\eps}{2M(\tilde{P}_k\left(\eps\right)+1)}\right)+\tilde{P}_k\left(\eps\right),\text{ where}\\
\tilde{P}_k\left(\eps\right)&=& \left\lceil\frac{12M^2(k+1)}{\eps}\Phi\left(\ds\frac{\eps}{12M(k+1)}\right)\right\rceil.
\eua

For all $k\ge 0$, let us denote

\bua
\chi^*_k(\eps)&:=& \chi_k(\eps/2),\\
\Theta_k(\eps)&:=  &\Theta(\eps,M^2,\theta,\chi^*_k) =  \theta\left(\chi^*_k(\eps/3)-1+\left\lceil\ln\left(\frac{3 M^2}{\eps}\right)\right\rceil\right)+1, \\
\Delta^*_k(\eps,g)&:= & \Delta(\eps,g,M^2,\theta,\chi^*_k)=\frac{\eps}{3g_{\eps,k}\left(\Theta_k(\eps)-\chi^*_k(\eps/3)\right)},
\eua
where $g_{\eps,k}(n)=n+g(n+\chi^*_k(\eps/3))$, $\Theta$ is defined by (\ref{def-Theta-quant-lemma-0}) and $\Delta$ by (\ref{def-Delta-quant-lemma-0}). Now let
\bua
f,f^*:\N\to\N,\,\, f(k):= \max\left\{\gamma\left(\frac{\Delta^*_k(\eps^2/4,g)}{M^2}\right), k\right\}-k, &
f^*(k) := f(k+\gamma(\eps_0))+\gamma(\eps_0).
\eua

We can apply Proposition \ref{quant-Browder-Halpern} for $\eps_0$ and $f^*$ to get the existence of $K_1\le K(\eps_0,f^*)$ such that for all $k,l\in [K_1,K_1+f^*(K_1)]$
\bea
d(z^u_k,z^u_l)\le \eps_0, \label{K1-dzuk,zul}
\eea
where $K$ is defined by (\ref{def-K-tk-1k+1}). Let 
\bua
K_0 &:=&  K_1+\gamma(\eps_0),\\
K^*(\eps_0,f) &:= &  K(\eps_0,f^*)+\gamma(\eps_0)=
\widetilde{f^*}^{(\lceil M^2/\eps_0^2\rceil)}(0)+\gamma(\eps_0),
\eua
with $\widetilde{f^*}(k):=k+f^*(k)$.

Then $\gamma(\eps_0)\le K_0\le K^*(\eps_0,f)$ and it is easy to see, using (\ref{K1-dzuk,zul}), that 
\beq
\forall k,l\in[K_0,K_0+f(K_0)]\,\,\big( d(z^u_k,z^u_l)\le \eps_0\big). \label{K0-meta-dzuk,zul}
\eeq
It follows that for all $k,l\in[K_0,K_0+f(K_0)]$,
\bua
d^2(u,Tz_k^u) &\le &  (d(u,Tz_l^u)+d(Tz_l^u,Tz_k^u))^2\le d^2(u,Tz_l^u)+d^2(Tz_l^u,Tz_k^u)+2d(u,Tz_l^u)d(Tz_l^u,Tz_k^u)\\
&\le & d^2(u,Tz_l^u)+d^2(z_l^u,z_k^u)+2Md(z_l^u,z_k^u)\\
&\le &   d^2(u,Tz_l^u)+\eps_0^2+2M\eps_0.
\eua

Let 
\beq
J:= K_0+f(K_0)=\max\left\{\gamma\left(\frac{\Delta^*_{K_0}(\eps^2/4,g)}{M^2}\right),K_0\right\}
\eeq

Then for all $n\ge 1$,
\bua
\gamma_n^J &=& \frac{J}{J+1}\,d^2(u,Tz_J^u)-d^2(x_{n+1},u)\le \frac{J}{J+1}
\big( d^2(u,Tz_{K_0}^u)+\eps_0^2+2M\eps_0\big)-d^2(x_{n+1},u) \\
&\le & d^2(u,Tz_{K_0}^u)-d^2(x_{n+1},u)+\eps_0^2+2M\eps_0 \\
&=&  \frac{{K_0}}{{K_0}+1}d^2(u,Tz_{K_0}^u)-d^2(x_{n+1},u)+\eps_0^2+2M\eps_0+ \frac{1}{K_0+1}d^2(u,Tz_{K_0}^u)\\
&=& \gamma_n^{K_0}+\eps_0^2+2M\eps_0+ \frac{1}{K_0+1}d^2(u,Tz_{K_0}^u)\le \gamma_n^{K_0}+\eps_0^2+2M\eps_0+ \frac{1}{K_0+1}M^2\\
&\le & \gamma_n^{K_0}+\eps_0^2+2M\eps_0+ M^2\eps_0 \qquad\text{as } K_0\ge \gamma(\eps_0)\\
&\le & \gamma_n^{K_0}+\frac{\eps^2}{24} \quad \text{by (\ref{eps0-prop}).}
\eua 
It follows that for all $n\ge \chi^*_{K_0}(\eps^2/12)$,
\bua
\gamma_n^J & \le &  \gamma_n^{K_0}+\frac{\eps^2}{24}\le \frac{\eps^2}{12}.
\eua
Applying (\ref{d2xn+1-zt-v3}) with $\ds t:=\frac{1}{J+1}$, we get that for all $n\ge 1$,
\bua
d^2(x_{n+1},z_J^u) &\le & (1-\lambda_{n+1})d^2(x_n,z_J^u)+\lambda_{n+1}\bigg(\frac{J}{J+1}\, d^2(u,Tz_J^u)-d^2(x_{n+1},u)\bigg)+\frac{M^2}{J+1}\\
&=& (1-\lambda_{n+1})d^2(x_n,z_J^u)+\lambda_{n+1}\gamma_n^J+\frac{M^2}{J+1}\\
&\le & (1-\lambda_{n+1})d^2(x_n,z_J^u)+\lambda_{n+1}\gamma_n^J+\Delta^*_{K_0}(\eps^2/4,g)
\eua
since $\ds J\ge \gamma\left(\frac{\Delta^*_{K_0}(\eps^2/4,g)}{M^2}\right)$, hence $\ds \frac{1}{J+1}\le \frac{\Delta^*_{K_0}(\eps^2/4,g)}{M^2}$.  It follows that we can apply Lemma \ref{quant-Aoyama-all-meta-rate-0} with $\eps:=\eps^2/4$ to conclude that for all $n\in [N,N+g(N)]$

\beq
d^2(x_{n},z_J^u) \le \frac{\eps^2}4,\quad\text{hence}\quad  d(x_{n},z_J^u)\le \frac{\eps}2,
\eeq
where $N:=\Theta_{K_0}(\eps^2/4)$.

Let now 
\bua
\theta^+(n)&:=& \max\{\theta(i)\mid i\le n\},\\
\Gamma &:=&\max \{\chi^*_k(\eps^2/12) \mid  \gamma(\eps_0)\le k\le 
K^*(\eps_0,f)\}\ge \chi^*_{K_0}(\eps^2/12),\\
\Sigma(\eps,g)&:= &\theta^+\left(\Gamma-1+\left\lceil\ln\left(\frac{12 M^2}{\eps^2}\right)\right\rceil\right)+1 \label{def-Sigma-proof} \\
&\ge & \theta^+\left(\chi^*_{K_0}(\eps^2/12)-1+\left\lceil\ln\left(\frac{12 M^2}{\eps^2}\right)\right\rceil\right)+1\\
&\ge &  \theta\left(\chi^*_{K_0}(\eps^2/12)-1+\left\lceil\ln\left(\frac{12 M^2}{\eps^2}\right)\right\rceil\right)+1\\
&=& \Theta_{K_0}(\eps^2/4) = N.
\eua

We get finally that $N\le \Sigma(\eps,g)$ is such that for all $n,m\in [N,N+g(N)]$,
\beq
d(x_n,x_m)\le  d(x_{n},z_J^u)+d(x_m,z_J^u)\le \eps.
\eeq
\hfill $\Box$

\mbox{ } 

\noindent
{\bf Acknowledgements:} \\[1mm] 
Ulrich Kohlenbach has been supported by the German Science Foundation (DFG Project KO 1737/5-1). Part of his research has been carried out while visiting the Simion Stoilow Institute of Mathematics of the 
Romanian Academy supported by the BITDEFENDER guest professor program. \\[1mm]
Lauren\c tiu Leu\c stean has been supported by a grant of the Romanian 
National Authority for Scientific Research, CNCS - UEFISCDI, project 
number PN-II-ID-PCE-2011-3-0383.

\end{document}